\documentclass[a4paper,11pt]{amsart}
\usepackage{amssymb}
\usepackage{mathrsfs}
\usepackage{amsmath,amssymb,amsthm,latexsym,amscd,mathrsfs}
\usepackage{indentfirst}
 \newcommand{\ROM}[1]{\mathrm{\uppercase\expandafter{\romannumeral#1}}}
\numberwithin{equation}{section} \theoremstyle{plain}
\newtheorem{thm}{Theorem}[section]

\newtheorem{lem}{Lemma}[section]

\theoremstyle{definition}

\theoremstyle{remark}
\newtheorem{rem}{Remark}[section]

\newtheorem{ack}{Acknowledgements}   

\title[DDVV-type inequality and Simons-type inequality]
{DDVV-type inequality for skew-symmetric matrices and Simons-type inequality for Riemannian submersions}
\author[J.Q. Ge]{Jianquan Ge}
\address{School of Mathematical Sciences, Laboratory of Mathematics and Complex Systems, Beijing Normal
University, Beijing 100875, P.R. CHINA} \email{jqge@bnu.edu.cn}
\thanks {The author is partially supported by the NSFC (No.11001016), the SRFDP (No. 20100003120003), and the Program for Changjiang Scholars and Innovative Research Team in University.}

\date{}
\begin{document}
\maketitle

 \begin{center}
 Dedicated to my advisor Professor Zizhou Tang.
 \end{center}
%%%%%%%%%%%%%%%%%%%%%
\begin{abstract}
In this paper, we will first
derive a DDVV-type optimal inequality for real skew-symmetric matrices,
then we apply it to establish a Simons-type integral inequality for
Riemannian submersions with totally geodesic fibres and Yang-Mills horizontal distributions. In this way, we show phenomenons of duality between Submanifold geometry and Riemannnian submersion, particularly between second fundamental form of a submanifold and integrability tensor of a Riemannian submersion.
\end{abstract}
%%%%%%%%%%%%%%%%%%%%%%%

%%%%%%%%%%%%%%%%%%%%%%
\section{Introduction}\label{sec1}
Let $M^n$ be an immersed submanifold of a real space form
$N^{n+m}(c)$ of constant sectional curvature $c$. Given an
orthonormal basis $\{e_1,\cdots,e_n\}$ (resp.
$\{\xi_1,\cdots,\xi_m\}$) of $T_pM$ (resp. $T^{\bot}_pM$), the
normalized scalar curvature $\rho$ and the normal scalar curvature
$\rho^{\bot}$ of $M^n$ at p are defined by
\[\rho=\frac{2}{n(n-1)}\sum^n_{1=i<j}\langle R(e_i,e_j)e_j,e_i\rangle,\]
\[\rho^{\bot}=\frac{2}{n(n-1)}\Big(\sum^n_{1=i<j}\sum^m_{1=r<s}\langle R^{\bot}(e_i,e_j)\xi_r, \xi_s\rangle^2\Big)^{\frac{1}{2}}=\frac{2}{n(n-1)}|R^{\bot}|,\]
where $R$ and $R^{\bot}$ are curvature tensors of the tangent and
normal bundles of $M$ respectively. Denote by $h$ the second
fundamental form and
$H=\frac{1}{n}Tr(h)=\frac{1}{n}\sum_{i=1}^nh(e_i,e_i)$ the mean
curvature vector field. The DDVV conjecture raised by \cite{DDVV}
says that there is a pointwise inequality among $\rho$,
$\rho^{\bot}$ and $|H|^2$ as the following:
\begin{equation}\label{DDVV ineq}
\rho+\rho^{\bot}\leq |H|^2+c.
\end{equation}
Due to the Gauss and Ricci equations, this conjecture can be
translated into the following algebraic inequality (cf. \cite{DFV}):
\begin{equation}\label{DDVV sym}
\sum_{r,s=1}^m\|[B_r,B_s]\|^2\leq \Big(\sum_{r=1}^m\|B_r\|^2\Big)^2,
\end{equation}
where $\{B_1, \cdots ,B_m\}$ are arbitrary real symmetric $(n\times
n)$-matrices, $[\cdot,\cdot]$ is the commutator operator and
$\|\cdot\|$ is the standard norm of matrix. In fact, putting $B_r=S_{\xi_r}-\langle H, \xi_r\rangle id$, where $S_{\xi_r}$ is the shape operator in direction $\xi_r$, we have
\begin{equation}\label{normal-shape}
|H|^2-\rho+c=\frac{1}{n(n-1)}\sum_{r=1}^m\|B_r\|^2,\quad\rho^{\bot}=\frac{1}{n(n-1)}\left(\sum_{r,s=1}^m\|[B_r,B_s]\|^2\right)^{\frac{1}{2}}.
\end{equation}

The inequality (\ref{DDVV sym}) (and thus the DDVV conjecture
(\ref{DDVV ineq})) has been proved independently and differently by
\cite{GT,Lu}. In particular, the equality condition given
in \cite{GT} shows that the inequality (\ref{DDVV sym}) is an
optimal inequality. As for the classification problem of
submanifolds attaining the equality of (\ref{DDVV ineq}) everywhere,
we refer to \cite{DT} for a big advance. In this paper, we obtain the following DDVV-type optimal inequality
for real skew-symmetric matrices which has been included as a part of the author's thesis and previously reviewed in the survey paper
\cite{GT2}.

Throughout this paper, a $K:=O(n)\times O(m)$ action on
$(B_1,\cdots,B_m)$ means that
$$(P,R)\cdot
(B_1,\cdots,B_m):=(PB_1P^t,\cdots,PB_mP^t)\cdot R,\quad for~~
(P,R)\in K.$$
\begin{thm}\label{thm1}
Let $B_1,\cdots,B_m$ be ($n\times n$) real skew-symmetric matrices.\\
$(i)$ If $n=3$, then we have \[\sum_{r,s=1}^m\|[B_r,B_s]\|^2\leq
\frac{1}{3}\Big(\sum_{r=1}^m\|B_r\|^2\Big)^2,\] where the equality
holds if and only if under some $K$ action all $B_r$'s are zero
except for $3$ matrices which can be written as
\[C_1:=\left(\begin{array}{ccc}0&
\lambda& 0\\-\lambda& 0& 0\\0& 0& 0
\end{array}\right),\quad C_2:=\left(\begin{array}{ccc}0&
0& \lambda\\0& 0& 0\\-\lambda& 0& 0
\end{array}\right),\quad C_3:=\left(\begin{array}{ccc}0& 0& 0\\0&
0& \lambda\\0& -\lambda& 0
\end{array}\right).\]
$(ii)$ If $n\geq 4$, then we have \[\sum_{r,s=1}^m\|[B_r,B_s]\|^2\leq
\frac{2}{3}\Big(\sum_{r=1}^m\|B_r\|^2\Big)^2,\] where the equality
holds if and only if under some $K$ action all $B_r$'s are zero
except for $3$ matrices which can be written as $diag(D_1, 0)$,
$diag(D_2, 0)$, $diag(D_3, 0)$, where $0\in M(n-4)$ is the zero
matrix of order $n-4$ and
\[D_1:=\left(\begin{array}{cccc}0& \lambda& 0&0\\-\lambda& 0& 0&0\\0&
0& 0&\lambda\\0&0&-\lambda&0
\end{array}\right),
D_2:=\left(\begin{array}{cccc}0& 0&\lambda& 0\\0& 0&
0&-\lambda\\-\lambda& 0& 0&0\\0&\lambda&0&0
\end{array}\right),
D_3:=\left(\begin{array}{cccc}0& 0&0&\lambda\\0& 0&\lambda&0\\0&
-\lambda& 0&0\\-\lambda&0&0&0
\end{array}\right).\]
\end{thm}
In sight of the geometric origin of the inequality (\ref{DDVV sym}),
\emph{i.e.}, the DDVV pointwise inequality (\ref{DDVV ineq}) in submanifold
geometry, we get interested in applications to geometry of this
``dual" algebraic inequality. In contrast to symmetric matrices, skew-symmetric
matrices have little geometric background, and much less for their commutators.
Finally, we focus on the geometry of Riemannian submersions which in some sense is also a
``dual" theory of submanifold geometry. It turns out rather
inspiring that, in analogy with symmetric matrices representing the second fundamental form of a submanifold,
skew-symmetric matrices can represent the integrability tensor of a Riemannian submersion. Therefore, one can trivially get an analogous DDVV-type pointwise inequality for Riemannian submersions once some similar ``normal scalar curvature" would be defined from commutators of these skew-symmetric matrices as in (\ref{normal-shape}).  Fortunately, motivated by the works of \cite{chern,Lu,Simons} where the symmetric
matrix inequality (\ref{DDVV sym}) takes an important role in the proof of the well-known Simons integral
inequality for closed minimal submanifolds in spheres, we find so for the skew-symmetric matrix
inequality in deducing a Simons-type integral inequality for Riemannian submersions with totally geodesic fibres and Yang-Mills horizontal distributions. In order to state the result
we first recall some notions about Riemannian submersions. The
notions in Chapter 9 of the book \cite{Be} will be used throughout this
paper.

Let $M^{n+m}$ and $B^n$ be (connected) Riemannian manifolds. A
smooth map $\pi: M\rightarrow B$ is called a \emph{Riemannian
submersion} if $\pi$ is of maximal rank and $\pi_{*}$ preserves the
lengths of horizontal vectors, \emph{i.e.}, vectors orthogonal to
the fibre $\pi^{-1}(b)$ for $b\in B$. Let $\mathscr{V}$ denote the
\emph{vertical distribution} consisting of vertical vectors (tangent
to the fibres) and $\mathscr{H}$ denote the \emph{horizontal
distribution} consisting of horizontal vectors on $M$. The
corresponding projections from $TM$ to $\mathscr{V}$ and
$\mathscr{H}$ are denoted by the same symbols. For Riemannian
submersions there are two fundamental tensors $T$ and $A$ on $M$
defined by O'Neill \cite{O} as follows. For vector fields $E_1$ and
$E_2$ on $M$,
\begin{equation}\label{A T}
\begin{array}{ll}
T_{E_1}E_2:=\mathscr{H}D_{\mathscr{V}E_1}\mathscr{V}E_2+\mathscr{V}D_{\mathscr{V}E_1}\mathscr{H}E_2,\\
A_{E_1}E_2:=\mathscr{H}D_{\mathscr{H}E_1}\mathscr{V}E_2+\mathscr{V}D_{\mathscr{H}E_1}\mathscr{H}E_2,&
\end{array}
\end{equation}
where $D$ is the Levi-Civita connection on $M$. In fact, $T$ is the
second fundamental form along each fibre if it is restricted to
vertical vectors, while $A$ measures the obstruction to
integrability of the horizontal distribution $\mathscr{H}$ and hence
it is called the \emph{integrability tensor} of $\pi$. Moreover,
some analogues of the Gauss-Codazzi equations for a Riemannian
submersion obtained by O'Neill \cite{O} are expressed in terms of
$T$ and $A$ as well as their covariant derivatives. These equations will be recovered in Section
\ref{simons-sec} by moving frame method, which is an effective method rarely used to the study of Riemannian submersions (cf. \cite{CG, Shen}) though widely adopted in submanifold geometry. More details
about $T$ and $A$ can be found in \cite{Be,O}.

Next we introduce the notion of \emph{Yang-Mills}
which has been intensely studied both in physics and in mathematics
and also found important for Einstein Riemannian submersions (see
for example \cite{At-Hi-Si,Be,Don,Tian} and references
therein). Here we use the presentation given in \cite{Be}. Let
$X_1,\cdots,X_n$ be a local orthonormal basis of the horizontal
distribution $\mathscr{H}$. Define a co-differential operator
$\check{\delta}$ over tensor fields on $M$ by
\begin{equation*}
\check{\delta}E:=-\sum_{i=1}^n(D_{X_i}E)_{X_i}.
\end{equation*}
Then we say that $\mathscr{H}$ satisfies the \emph{Yang-Mills}
condition if, for any vertical vector $U$ and any horizontal vector
$X$, we have
\begin{equation*}
\langle\check{\delta}A(X),U\rangle-\langle A_X,T_U\rangle=0,
\end{equation*}
where the bracket $\langle\cdot,\cdot\rangle$ denotes the metric of $M$ and also its induced metric on tensors. As pointed out in
\cite{Be}, this condition depends only on $\mathscr{H}$ and the
metric of $B$ and not on the family of metrics on the fibres. When the
fibres are totally geodesic, \emph{i.e.}, $T=0$, this condition is
equivalent to
$$\check{\delta}A=0.$$ And in this case, it is one of the three sufficient and
necessary conditions for $M$ to be Einstein (see (\ref{ricci-curv})). Furthermore, motivated by the equation (\ref{einsteinM}), we can also regard this condition as a dual of the minimality condition, or its first derivative, for a submanifold in a sphere. A natural interaction between Yang-Mills connections and minimal submanifolds has been investigated by Tian \cite{Tian}.

To be coherent with that in \cite{Be}, we define the square norm of $A$ by
\begin{equation}\label{A norm}
|A|^2:=\sum_{i,j=1}^n\langle A_{X_i}X_j,A_{X_i}X_j\rangle=\sum_{i=1}^n\sum_{r=1}^m\langle A_{X_i}U_r,A_{X_i}U_r\rangle,
\end{equation}
where $\{U_1,\cdots,U_m\}$ is a local orthonormal basis of the
vertical distribution $\mathscr{V}$. This invariant is just our target in the integrand of the Simons-type integral inequality corresponding to the square norm of the second fundamental form in the original Simons inequality in submanifold geometry. Besides several references
cited in \cite{Be}, it is noteworthy that this invariant has
been also studied by Chen (\cite{Chen}, \emph{etc.}) who denoted it
by $\breve{A}_{\pi}$ and obtained its sharp upper bound for an
arbitrary isometric immersion from $M$ (with totally geodesic
fibres) into a unit sphere in terms of square norm of the mean
curvature of the immersion.

Now we are ready to state the Simons-type integral inequality and some equality characterizations as follows.
For $x\in M$, we denote by $\check{\kappa}(x)$ the largest eigenvalue of
the curvature operator $\check{R}:\bigwedge^2TB\rightarrow\bigwedge^2TB$ of $B$ at $\pi(x)\in B$,
$\check{\lambda}(x)$ the lowest eigenvalue of the Ricci curvature
$\check{r}$ of $B$ at $\pi(x)\in B$ (thus $\check{\kappa}$, $\check{\lambda}$ are constant along any fibre), and $\hat{\mu}(x)$ the
largest eigenvalue of the Ricci curvature $\hat{r}$ of the fibre at
$x$.
\begin{thm}\label{Thm-simonstype ineq}
Let $\pi: M^{n+m}\rightarrow B^n$ be a Riemannian submersion with
totally geodesic fibres and Yang-Mills horizontal distribution, i.e., $T=0$ and $\check{\delta}A=0$.
Suppose that $M$ is closed. Then the following cases hold:
\begin{itemize}
\item[(i)] If $n=2$, then we have
\begin{equation*}
\int_M~|A|^2\hat{\mu}~dV_M\geq0;
\end{equation*}
\item[(ii)] If $m=1$, then we have
 \begin{equation*}
  \int_M~|A|^2(\check{\kappa}-\check{\lambda})~dV_M\geq0;
\end{equation*}
\item[(iii)] If $m\geq2$ and $n=3$, then we have
 \begin{equation*}
  \int_M~|A|^2(\frac{1}{6}|A|^2+2\hat{\mu}+\check{\kappa}-\check{\lambda})~dV_M\geq0;
\end{equation*}
\item[(iv)] If $m\geq2$ and $n\geq4$, then we have
 \begin{equation*}
  \int_M~|A|^2(\frac{1}{3}|A|^2+2\hat{\mu}+\check{\kappa}-\check{\lambda})~dV_M\geq0.
\end{equation*}
\end{itemize}
 Moreover, if $A\neq0$, or equivalently, $M$ is not locally a
Riemannian product $B\times F$, then we have the following
conclusions about the equality conditions:
\begin{itemize}
\item[(a)] In each case, if the equality holds, then each fibre has flat normal
bundle in $M$ and $|A|^2\equiv Const=:C>0$, which implies further
the following:
\begin{itemize}
\item[(a1)] In case (i),  $\hat{\mu}\equiv0$;
\item[(a2)] In case (ii), $\check{\kappa}-\check{\lambda}\equiv0$;
\item[(a3)] In case (iii), $\hat{\mu}\equiv\frac{1}{12}C$,
$\check{\kappa}-\check{\lambda}\equiv\frac{-1}{3}C$;
\item[(a4)] In case (iv), $\hat{\mu}\equiv\frac{1}{6}C$,
$\check{\kappa}-\check{\lambda}\equiv\frac{-2}{3}C$.
\end{itemize}
\item[(b)] If the equality in (iii) or (iv) holds, then $m\geq3$ and at each point of $M$
there exist an orthonormal vertical basis $\{U_1,\cdots,U_m\}$ and
an orthonormal horizontal basis $\{X_1,\cdots,X_n\}$ such that the
$(n\times n)$ skew-symmetric matrices
$$A^r:=\Big(\langle A_{X_i}U_r,X_j\rangle\Big)_{n\times n},\quad r=1,\cdots,m,$$
are in the forms of the matrices in the equality conditions of (i)
or (ii) of Theorem \ref{thm1} respectively. Furthermore, under
these basis, the following decompositions hold
\begin{eqnarray}
&&\hat{r}=\hat{\mu}I_3\oplus \hat{r}', \nonumber\\
&&\check{R}\equiv\check{\kappa}I_3,\quad
\check{r}\equiv2\check{\kappa}I_3,\quad in~~case~~(iii),\nonumber\\
&&\check{R}=\check{\kappa}I_6\oplus \check{R}',\quad
\check{r}\equiv\check{\lambda}I_4\oplus\check{r}',\quad
in~~case~~(iv),\nonumber
\end{eqnarray}
where $\hat{r}'=\hat{r}|_{span\{U_4,\cdots,U_m\}}$,
$\check{R}'=\check{R}|_{span\{X_i\wedge X_j|1\leq i\leq n,~5\leq
j\leq n\}}$, and $\check{r}'=\check{r}|_{span\{X_5,\cdots,X_n\}}$.
In particular, when $m=3$, the fibres have constant
sectional curvature. Similarly, when $3\leq n\leq5$, the base manifold $B^n$
has constant sectional curvature. More precisely and specifically,
we have the following (c-d).
\item[(c)] When $m=3$, if the equality in (iii) holds, then there
exist some $a>0$ such that
\begin{itemize}
\item[(c1)] all fibres are isometric to a manifold $F^3$ of constant
sectional curvature $a$;
\item[(c2)] the base manifold $B^3$ has constant sectional curvature $8a$;
\item[(c3)] the following identities hold:
\begin{eqnarray}
&&|A|^2\equiv24a,\nonumber\\
&&K_{rs}\equiv a,\quad K_{ij}\equiv-4a,\quad
K_{ir}=\Big\{\begin{array}{ll} 0&
for~~(i,r)=(1,3),(2,2),(3,1)\\4a&otherwise,
\end{array}\nonumber\\
&&R_{rs}\equiv10a\delta_{rs},\quad R_{ij}\equiv0,\quad
R_{ir}\equiv0,\nonumber
\end{eqnarray}
where $K_{rs}$, $K_{ij}$, $K_{ir}$ (resp. $R_{rs}$, $R_{ij}$,
$R_{ir}$) are sectional curvatures (resp. Ricci curcatures) of $M$
on the $2$-planes spanned by $\{U_r,U_s\}$, $\{X_i,X_j\}$,
$\{X_i,U_r\}$, respectively, under the basis $\{U_r\}$ and $\{X_i\}$
given in case (b).
\end{itemize}
\item[(d)] When $m=3$, if the equality in (iv) holds, then there
exist some $a>0$ such that all fibres are isometric to a manifold
$F^3$ of constant sectional curvature $a$. In addition,
\begin{itemize}
\item[(d1)] if $n=4$, then the submersion $\pi$ is covered by the
Hopf fibration $\pi_0:S^7(\frac{1}{\sqrt{a}})\rightarrow
S^4(\frac{1}{2\sqrt{a}})$, i.e., there are two covering maps
$\pi_1: S^7(\frac{1}{\sqrt{a}})\rightarrow M^7$ and $\pi_2: S^4(\frac{1}{2\sqrt{a}})\rightarrow B^4$  such that
$\pi_2\circ\pi_0=\pi\circ\pi_1$;
\item[(d2)] if $n=5$, then the base manifold $B^5$ has constant sectional
curvature $\frac{8}{3}a$, and the following identities hold (with the
same notations as in (c3)):
\begin{eqnarray}
&&|A|^2\equiv12a,\nonumber\\
&&K_{rs}\equiv a,\quad
K_{ij}=\Big\{\begin{array}{ll}
\frac{-1}{3}a & for~~1\leq i< j\leq4\\
\frac{8}{3}a & for~~1\leq i< j=5,
\end{array}\quad
K_{ir}=\Big\{\begin{array}{ll}
a & for~~1\leq i\leq4\\
0 & for~~i=5,
\end{array}\nonumber\\
&&R_{rs}\equiv6a\delta_{rs},\quad
R_{ij}=\Big\{\begin{array}{ll}
\frac{14}{3}a\delta_{ij} & for~~1\leq i,j\leq4\\
\frac{32}{3}a\delta_{ij} & for~~1\leq i\leq j=5,
\end{array}\quad
R_{ir}\equiv0.\nonumber
\end{eqnarray}
\end{itemize}
\end{itemize}
\end{thm}

\begin{rem}
As we mentioned previously, the Yang-Mills condition is implied by the
Einstein condition of $M$ when the fibres are totally geodesic.
Therefore, examples satisfying our assumptions of the theorem are
plentiful (cf. \cite{Be}). Note that the corresponding pointwise
inequalities with the same equality conclusions also hold when $M$
is not closed, provided that $|A|^2$ is constant on $M$, which is also
a condition implied by the Einstein condition of $M$.
\end{rem}
\begin{rem}
Besides the classification problem, searching examples of Riemannian
submersions in (c) and (d2) of the theorem might make sense to itself. For instance, if $M^6$ is simply
connected with connected fibres in case (c), then $B^3$ and $F^3$ are round spheres and $M^6=B^3\times F^3$
is a topological product but not Riemannian, nor warped product (cf.
Remark 9.57 in \cite{Be}, and \cite{He,Na}).
\end{rem}

To conclude the introduction, we remark that as the Chern
problem, the Peng-Terng gap theorem and the classification problem of its equality case, all based on the Simons inequality in submanifold
geometry (cf. \cite{chern, CCK,DX, GT3,Law,Lu,Peng-T1,Peng-T2, Yau}, \emph{etc.}), one can now ask the ``dual"
versions for Riemannian submersions with square norm of the
integrability tensor $A$ instead of square norm of the second
fundamental form $h$.

\section{DDVV-type skew-symmetric matrix inequality}
\subsection{Notations and preparing lemmas} Throughout this section, we
denote by $M(m,n)$ the space of $m\times n$ real matrices, $M(n)$
the space of $n\times n$ real matrices and $\mathfrak{o}(n)$ the
$N:=\frac{n(n-1)}{2}$ dimensional subspace of skew-symmetric
matrices in $M(n)$. \vskip 0.05cm For every $(i,j)$ with $1\leq
i<j\leq n$, let $\tilde{E}_{ij}:=\frac{1}{\sqrt{2}}(E_{ij}-E_{ji})$,
where $E_{ij}\in M(n)$ is the matrix with $(i,j)$ entry $1$ and all
others $0$. Clearly $\{\tilde{E}_{ij}\}_{i<j}$ is an orthonormal
basis of $\mathfrak{o}(n)$. Let us take an order of the indices set
$S:=\{(i,j)| 1\leq i< j\leq n\}$ by
\begin{equation}\label{order}
(i,j)<(k,l) \hskip 0.2cm if\hskip 0.1cm and\hskip 0.1cm only\hskip
0.1cm if \hskip 0.2cm i<k\hskip 0.1cm or\hskip 0.1cm i=k<j<l.
\end{equation}
In this way we can identify $S$ with $\{1,\cdots,N\}$ and write
elements of $S$ in Greek, i.e. for $\alpha=(i,j)\in S$, we can say
$1\leq\alpha\leq N$.

 For $\alpha=(i,j)<(k,l)=\beta$ in $S$, direct
calculations show that
\begin{equation}\label{E tilde norm}
\|[\tilde{E}_{\alpha}, \tilde{E}_{\beta}]\|^2=
\begin{cases}
\frac{1}{2}, \quad i<j=k<l ~ or ~
i=k<j<l ~ or ~ i<k<j=l;\\
0, \quad otherwise,
\end{cases}
\end{equation}
and for any $\alpha, \beta\in S,$
\begin{equation}\label{E tilde inner}
\sum_{\gamma\in S}~ \langle~ [\tilde{E}_{\alpha},
\tilde{E}_{\gamma}],~ [\tilde{E}_{\beta}, \tilde{E}_{\gamma}]~
\rangle= (n-2)\delta_{\alpha\beta},
\end{equation}
where $\delta_{\alpha\beta}=\delta_{ik}\delta_{jl}$, and $\langle
\cdot , \cdot\rangle$ is the standard inner product of $M(n)$.

Let $\{\tilde{Q}_{\alpha}\}_{\alpha\in S}$ be any orthonormal basis
of $\mathfrak{o}(n)$. There exists a unique orthogonal matrix $Q\in
O(N)$ such that
$(\tilde{Q}_1,\cdots,\tilde{Q}_N)=(\tilde{E}_1,\cdots,\tilde{E}_N)Q$,
i.e.
$\tilde{Q}_{\alpha}=\sum_{\beta}q_{\beta\alpha}\tilde{E}_{\beta}$
for $Q=(q_{\alpha\beta})_{N\times N}$. If we set
$\tilde{Q}_{\alpha}=(\tilde{q}^{\alpha}_{ij})_{n\times n}$, then
$\tilde{q}^{\alpha}_{ij}=-\tilde{q}^{\alpha}_{ji}=
\frac{1}{\sqrt{2}}q_{\beta\alpha}$ for $\beta=(i,j)\in S$. Henceforth,
this correspondence between an orthonormal basis $\{\tilde{Q}_{\alpha}\}_{\alpha\in S}$ of $\mathfrak{o}(n)$ and an orthogonal matrix $Q\in O(N)$ is regarded known.

Let $\lambda_1,\cdots,\lambda_{[\frac{n}{2}]}$ be $[\frac{n}{2}]$
real numbers satisfying $\sum_i\lambda_i^2=\frac{1}{2}$ and
$\lambda_1\geq \cdots\geq \lambda_{[\frac{n}{2}]}\geq 0$. Denote by
$I:=\{(i,j)\in S|(\lambda_i+\lambda_j)^2>\frac{2}{3}\}$ and $n_0$
the number of elements of $I$. It is easily seen that $n_0=0$ when
$n=3$. Moreover, we have
\begin{lem}\label{lem1}
If $I$ is not empty, i.e. $n_0\geq 1$, then
\[I=\{1\}\times \{2,\cdots,n_0+1\}, \quad n_0+1\leq [\frac{n}{2}].\]
\end{lem}
\begin{proof}
Obviously, by the assumptions of $\lambda_i$'s, $(1,2)\in I$ if $I$
is not empty. It suffices to prove that $(2,3)$ is not in $I$.
Otherwise, we have
\[(\lambda_1+\lambda_2)^2\geq(\lambda_1+\lambda_3)^2\geq (\lambda_2+\lambda_3)^2>\frac{2}{3},\]
and thus
\[4(\lambda_1^2+\lambda_2^2+\lambda_3^2)\geq(\lambda_1+\lambda_2)^2+(\lambda_1+\lambda_3)^2+(\lambda_2+\lambda_3)^2>2,\]
which contradicts with
$\lambda_1^2+\lambda_2^2+\lambda_3^2\leq\sum_i\lambda_i^2=\frac{1}{2}.$
\end{proof}

\begin{lem}\label{lem2}
 We have
\[\sum_{(i,j)\in I}\Big((\lambda_i+\lambda_j)^2-\frac{2}{3}\Big)\leq \frac{1}{3},\]
where the equality holds if and only if $n_0=1$,
$\lambda_1=\lambda_2=\frac{1}{2}$ and all other $\lambda_j$'s 0.
\end{lem}
\begin{proof}
By Lemma \ref{lem1},
\begin{eqnarray}
\sum_{(i,j)\in I}[(\lambda_i+\lambda_j)^2-\frac{2}{3}]
&=&\sum_{j=2}^{n_0+1}(\lambda^2_1+\lambda^2_j+2\lambda_1\lambda_j)-\frac{2}{3}n_0\nonumber\\
&=&n_0\lambda^2_1+\sum_{j=2}^{n_0+1}\lambda^2_j+2\lambda_1\sum_{j=2}^{n_0+1}\lambda_j-\frac{2}{3}n_0\nonumber\\
&\leq&(n_0+1)\lambda^2_1+\sum_{j=2}^{n_0+1}\lambda^2_j+\Big(\sum_{j=2}^{n_0+1}\lambda_j\Big)^2-\frac{2}{3}n_0\nonumber\\
&\leq&(n_0+1)\Big(\lambda^2_1+\sum_{j=2}^{n_0+1}\lambda^2_j\Big)-\frac{2}{3}n_0\nonumber\\
&\leq&(n_0+1)\sum_{i}\lambda^2_i-\frac{2}{3}n_0=
\frac{n_0+1}{2}-\frac{2}{3}n_0\leq \frac{1}{3},\nonumber
\end{eqnarray}
where the equality condition is easily seen from the proof.
\end{proof}

\begin{lem}\label{lem3}
For any $Q\in O(N)$, $\alpha\in S$ and any subset $J_{\alpha}\subset S$, we
have
\[\sum_{\beta\in J_{\alpha}}\Big( \|[\tilde{Q}_{\alpha},
\tilde{Q}_{\beta}]\|^2-\frac{2}{3}\Big)\leq \frac{2}{3}.\]
\end{lem}
\begin{proof}
Given $\alpha\in S$, under some $O(n)\subset K$ action, without loss
of generality, we can assume
$$\tilde{Q}_{\alpha}=diag\left(\left(\begin{array}{cc}0& \lambda_1
\\-\lambda_1& 0
\end{array}\right),\cdots,\left(\begin{array}{cc}0& \lambda_{[\frac{n}{2}]}
\\-\lambda_{[\frac{n}{2}]}& 0
\end{array}\right),0\right),$$
where $\lambda_1\geq\cdots\geq\lambda_{[\frac{n}{2}]}\geq 0$,
$\sum_i\lambda^2_i=\frac{1}{2}$ and the last $0$ exists only if $n$
is odd.

Put \begin{equation}\label{U matrix}
U:=diag\left(\left(\begin{array}{cc}
\frac{1}{\sqrt{2}}&\frac{\sqrt{-1}}{\sqrt{2}}
\\\frac{\sqrt{-1}}{\sqrt{2}}& \frac{1}{\sqrt{2}}
\end{array}\right),\cdots,\left(\begin{array}{cc}
\frac{1}{\sqrt{2}}&\frac{\sqrt{-1}}{\sqrt{2}}
\\\frac{\sqrt{-1}}{\sqrt{2}}& \frac{1}{\sqrt{2}}
\end{array}\right),1\right),
\end{equation}
 where the last $1$ exists only if $n$ is
odd. Set $\check{Q}_{\gamma}:=U\sqrt{-1}\tilde{Q}_{\gamma}U^*:=(\check{q}^{\gamma}_{ij})$ for $\gamma\in S$, where $U^*$ denotes the conjugate transpose. Then the following identities can be easily verified for $k,l=1,\cdots,[\frac{n}{2}]$ and $k<l$:
\begin{eqnarray*}
&&\check{q}^{\gamma}_{2k-1,2k-1}=-\check{q}^{\gamma}_{2k,2k}=\tilde{q}^{\gamma}_{2k-1,2k},\quad \check{q}^{\gamma}_{n,n}=0\quad if~n~is~odd;\\
&&\check{q}^{\gamma}_{2k-1,2k}=\check{q}^{\gamma}_{2k,2k-1}=0;\\
&&\check{q}^{\gamma}_{2k-1,2l-1}=-\overline{\check{q}^{\gamma}_{2k,2l}}=\frac{1}{2}\{(\tilde{q}^{\gamma}_{2k-1,2l}-\tilde{q}^{\gamma}_{2k,2l-1})+\sqrt{-1}(\tilde{q}^{\gamma}_{2k-1,2l-1}+\tilde{q}^{\gamma}_{2k,2l})\};\\
&&\check{q}^{\gamma}_{2k-1,2l}=-\overline{\check{q}^{\gamma}_{2k,2l-1}}=\frac{1}{2}\{(\tilde{q}^{\gamma}_{2k-1,2l-1}-\tilde{q}^{\gamma}_{2k,2l})+\sqrt{-1}(\tilde{q}^{\gamma}_{2k,2l-1}+\tilde{q}^{\gamma}_{2k-1,2l})\};\\
&&\check{q}^{\gamma}_{2k-1,n}=\sqrt{-1}\overline{\check{q}^{\gamma}_{2k,n}}=\overline{\check{q}^{\gamma}_{n,2k-1}}=-\sqrt{-1}\check{q}^{\gamma}_{n,2k}=\frac{1}{\sqrt{2}}(-\tilde{q}^{\gamma}_{2k,n}+\sqrt{-1}\tilde{q}^{\gamma}_{2k-1,n})~~ if~n~is~odd.
\end{eqnarray*}
In particular,
$$\check{Q}_{\alpha}=
diag(\lambda_1,-\lambda_1,\cdots,\lambda_{[\frac{n}{2}]},-\lambda_{[\frac{n}{2}]},0)=:diag(u_1,u_2,\cdots,u_n).$$
For any $(i,j)\in
\acute{S}:=\{(i,j)\in S|(i,j)\neq (2k-1,2k), 1\leq k\leq
[\frac{n}{2}]\}$, it follows from the identities above that
$$\sum_{\gamma\in S}|\check{q}^{\gamma}_{ij}|^2=\frac{1}{2}.$$
As for the proof, we take $(2k-1,2l-1)\in\acute{S}$ for example:
\begin{eqnarray*}
\sum_{\gamma\in S}|\check{q}^{\gamma}_{2k-1,2l-1}|^2
&=&\sum_{\gamma\in S}\frac{1}{4}\Big((\tilde{q}^{\gamma}_{2k-1,2l})^2+(\tilde{q}^{\gamma}_{2k,2l-1})^2+(\tilde{q}^{\gamma}_{2k-1,2l-1})^2+(\tilde{q}^{\gamma}_{2k,2l})^2\\
&&-2\tilde{q}^{\gamma}_{2k-1,2l}\tilde{q}^{\gamma}_{2k,2l-1}+2\tilde{q}^{\gamma}_{2k-1,2l-1}\tilde{q}^{\gamma}_{2k,2l}\Big)\\
&=&\sum_{\gamma\in S}\frac{1}{8}\Big((q_{(2k-1,2l)\gamma})^2+(q_{(2k,2l-1)\gamma})^2+(q_{(2k-1,2l-1)\gamma})^2+(q_{(2k,2l)\gamma})^2\\
&&-2q_{(2k-1,2l)\gamma}q_{(2k,2l-1)\gamma}+2q_{(2k-1,2l-1)\gamma}q_{(2k,2l)\gamma}\Big)\\
&=&\frac{1}{8}(1+1+1+1+0+0)=\frac{1}{2}.
\end{eqnarray*}
 Denote by $\check{S}:=\{(i,j)\in
\acute{S}|(u_i-u_j)^2>\frac{2}{3}\}$. Since $\sum_i\lambda^2_i=\frac{1}{2}$, we find that $u_iu_j<0$ for $(i,j)\in\check{S}$ and hence $(u_i,u_j)=(\lambda_k,-\lambda_l)$ or $(-\lambda_k,\lambda_l)$ for some $(k,l)\in I$. Then by the preceding identities and Lemma \ref{lem2}, we complete the proof of the lemma as follows:
\begin{eqnarray}
\sum_{\beta\in J_{\alpha}}\Big( \|[\tilde{Q}_{\alpha},
\tilde{Q}_{\beta}]\|^2-\frac{2}{3}\Big)&=&\sum_{\beta\in
J_{\alpha}}\Big(
\|[\check{Q}_{\alpha}, \check{Q}_{\beta}]\|^2-\frac{2}{3}\Big)\nonumber\\
&=&\sum_{\beta\in
J_{\alpha}}\sum^n_{i,j=1}\Big((u_i-u_j)^2-\frac{2}{3}\Big)|\check{q}^{\beta}_{ij}|^2\nonumber\\
&\leq&\sum_{\beta\in
J_{\alpha}}2\sum_{i<j}\Big((u_i-u_j)^2-\frac{2}{3}\Big)|\check{q}^{\beta}_{ij}|^2\nonumber\\
&=&2\sum_{\beta\in
J_{\alpha}}\sum_{(i,j)\in \acute{S}}\Big((u_i-u_j)^2-\frac{2}{3}\Big)|\check{q}^{\beta}_{ij}|^2\nonumber\\
&\leq&2\sum_{(i,j)\in
\check{S}}\Big((u_i-u_j)^2-\frac{2}{3}\Big)\sum_{\beta\in
J_{\alpha}}|\check{q}^{\beta}_{ij}|^2\nonumber\\
&\leq&2\sum_{(i,j)\in
\check{S}}\Big((u_i-u_j)^2-\frac{2}{3}\Big)\sum_{\beta\in S}|\check{q}^{\beta}_{ij}|^2\nonumber\\
&\leq& 4\sum_{(k,l)\in
I}\Big((\lambda_k+\lambda_l)^2-\frac{2}{3}\Big)\frac{1}{2}\nonumber\\
&\leq&\frac{2}{3}.\nonumber
\end{eqnarray}
\end{proof}

\begin{lem}\label{lem4}
For any $Q\in O(N)$ and $\alpha\in S$, we have
$$\sum_{\beta\in
S}\|[\tilde{Q}_{\alpha}, \tilde{Q}_{\beta}]\|^2=n-2.$$
\end{lem}
\begin{proof}
It follows from $(\ref{E tilde inner})$ that
\begin{eqnarray}
\sum_{\beta\in S}\|[\tilde{Q}_{\alpha},
\tilde{Q}_{\beta}]\|^2&=&\sum_{\beta\gamma\tau\xi\eta}q_{\gamma\alpha}q_{\xi\alpha}q_{\tau\beta}q_{\eta\beta}\langle~[\tilde{E}_{\gamma},
\tilde{E}_{\tau}],~
[\tilde{E}_{\xi}, \tilde{E}_{\eta}]~ \rangle\nonumber\\
&=&\sum_{\gamma\tau\xi\eta}q_{\gamma\alpha}q_{\xi\alpha}\delta_{\tau\eta}\langle~[\tilde{E}_{\gamma},
\tilde{E}_{\tau}],~
[\tilde{E}_{\xi}, \tilde{E}_{\eta}]~ \rangle\nonumber\\
&=&\sum_{\gamma\xi}q_{\gamma\alpha}q_{\xi\alpha}\sum_{\tau}\langle~
[\tilde{E}_{\gamma}, \tilde{E}_{\tau}],~
[\tilde{E}_{\xi}, \tilde{E}_{\tau}]~ \rangle\nonumber\\
&=&\sum_{\gamma\xi}q_{\gamma\alpha}q_{\xi\alpha}(n-2)\delta_{\gamma\xi}
=(n-2)\sum_{\gamma}q^2_{\gamma\alpha}= n-2.\nonumber
\end{eqnarray}
\end{proof}

\begin{lem}\label{lem5}
Let $A,B$ be $(n\times n)$ real skew-symmetric matrices.\\
$(i)$ If $n=3$, then we have \[\|[A,B]\|^2\leq\frac{1}{2}\|A\|^2\|B\|^2,\]
where the equality holds if and only if there is a $P\in O(3)$ such
that \[PAP^t=C_1,\quad PBP^t=aC_2+bC_3,\] where $C_1,C_2,C_3$ are
the matrices in Theorem \ref{thm1} and $a,b$ are real numbers.\\
$(ii)$ If $n\geq4$, then we have
\[\|[A,B]\|^2\leq\|A\|^2\|B\|^2,\]
where the equality holds if and only if there is a $P\in O(n)$ such
that \[PAP^t=diag(D_1,0),\quad PBP^t=a\cdot diag(D_2,0)+b\cdot
diag(D_3,0),\] where $D_1,D_2,D_3$ are the matrices in Theorem
\ref{thm1} and $a,b$ are real numbers.
\end{lem}
\begin{proof}
$(i)$ As $A$ is now a $(3\times 3)$ real skew-symmetric matrix,
there is a $P\in O(3)$ such that \[PAP^t=\left(\begin{array}{ccc}0&
\lambda& 0\\-\lambda& 0& 0\\0& 0& 0
\end{array}\right)=C_1.\]
Denote by $PBP^t:=(b_{ij})\in \mathfrak{o}(3)$. Then direct
computation shows that
\[[PAP^t,PBP^t]=\left(\begin{array}{ccc}0&
0& \lambda b_{23}\\0& 0& -\lambda b_{13}\\ -\lambda b_{23}&\lambda
b_{13} & 0
\end{array}\right).\]
Thus
\[\|[A,B]\|^2=\|[PAP^t,PBP^t]\|^2=2\lambda^2(b_{23}^2+b_{13}^2)\leq\frac{1}{2}\|A\|^2\|B\|^2,\]
where the equality holds if and only if $b_{12}=0$, \emph{i.e.},
$PBP^t$ lies in $Span\{C_2,C_3\}$.\\
$(ii)$ As $A$ is now a $(n\times n)$ real skew-symmetric matrix,
there is a $P\in O(n)$ such that
\[PAP^t=diag\left(\left(\begin{array}{cc}0& \lambda_1
\\-\lambda_1& 0
\end{array}\right),\cdots,\left(\begin{array}{cc}0& \lambda_{[\frac{n}{2}]}
\\-\lambda_{[\frac{n}{2}]}& 0
\end{array}\right),0\right),\]
where $\lambda_1\geq\cdots\geq\lambda_{[\frac{n}{2}]}\geq 0$ and the
last $0$ exists only if $n\geq4$ is odd.

Let $U$ be the unitary matrix defined in (\ref{U matrix}). Then we
have
\[\check{A}:=U\sqrt{-1}PAP^tU^{*}=
diag(\lambda_1,-\lambda_1,...,\lambda_{[\frac{n}{2}]},-\lambda_{[\frac{n}{2}]},0):=diag(u_1,u_2,...,u_n).\]
Put
\[
\check{B}:=U\sqrt{-1}PBP^tU^{*}:=(b_{ij}),\quad sgn(n)=\Big\{
\begin{array}{ll}
1 & for~~n~odd,\\
0 & for~~n~even.
\end{array}
\]
 Then it follows from the proof of Lemma \ref{lem3} that $b_{2k-1,2k}=0$ and
 \begin{eqnarray}
\|[A,B]\|^2&=&\|[\check{A},\check{B}]\|^2=\sum_{i,j=1}^n(u_i-u_j)^2|b_{ij}|^2\nonumber\\
&=&2\Big(\sum_{k<l}[(\lambda_k-\lambda_l)^2(|b_{2k-1,2l-1}|^2+|b_{2k,2l}|^2)+(\lambda_k+\lambda_l)^2
(|b_{2k-1,2l}|^2+|b_{2k,2l-1}|^2)]\Big)\nonumber\\
&~&+2(sgn(n))\sum_k\lambda_k^2(|b_{2k-1,n}|^2+|b_{2k,n}|^2)\nonumber\\
&\leq&2\Big(\sum_{k<l}(\lambda_1+\lambda_2)^2(|b_{2k-1,2l-1}|^2+|b_{2k,2l}|^2+|b_{2k-1,2l}|^2+|b_{2k,2l-1}|^2)\Big)
\nonumber\\
&~&+2(sgn(n))\sum_k\lambda_1^2(|b_{2k-1,n}|^2+|b_{2k,n}|^2)\nonumber\\
&\leq&2\|A\|^2\Big(\sum_{k<l}(|b_{2k-1,2l-1}|^2+|b_{2k,2l}|^2+|b_{2k-1,2l}|^2+|b_{2k,2l-1}|^2)\Big)\nonumber\\
&~&+\|A\|^2(sgn(n))\sum_k(|b_{2k-1,n}|^2+|b_{2k,n}|^2)\nonumber\\
&\leq&\|A\|^2\|B\|^2.\nonumber
\end{eqnarray}
Analyzing these inequalities, we find that the equality in this case
holds if and only if $\lambda_1=\lambda_2=\frac{1}{2}\|A\|$,
$\lambda_j=0$ for $j>2$, and all $b_{ij}$'s are zero except
$b_{14}=\bar{b}_{41}$ and $b_{23}=\bar{b}_{32}$, which is equivalent
to that $PAP^t, PBP^t$ are in the forms specified in the lemma.
\end{proof}

Now let $\varphi : M(m,n)\longrightarrow M(C_m^2,C_n^2)$ be the map
defined by $\varphi(A)_{(i,j)(k,l)}:=A(_{i\hskip 0.1cm j}^{k\hskip
0.1cm l})$, where $C_m^2=\frac{m(m-1)}{2}$, $1\leq i<j\leq m$, $1\leq k<l\leq n$ and
$A(_{i\hskip 0.1cm j}^{k\hskip 0.1cm l})=a_{ik}a_{jl}-a_{il}a_{jk}$
is the determinant of the sub-matrix of $A:=(a_{ij})$ with the rows $i, j$,
the columns $k, l$, arranged with the same ordering as in
(\ref{order}). It is easily seen that $\varphi(I_n)=I_{C_n^2}$ (preserving identity matrices),
$\varphi(A)^t=\varphi(A^t)$ and the following
\begin{lem}\label{lem6}
 The map $\varphi$ preserves the matrix product, i.e.
 $\varphi(AB)=\varphi(A)\varphi(B)$ holds for $A\in M(m,k)$, $B\in
M(k,n)$.
\end{lem}

We will also need the following exercise of linear algebra in the
proof of the equality case of Theorem \ref{thm1}.
\begin{lem}\label{lem7}
 Let $A$, $B$ be two matrices in $M(m,n)$. Then
$AA^t=BB^t$ if and only if $A=BR$ for some $R\in O(n)$.
\end{lem}
\subsection{Proof of Theorem \ref{thm1}} Let $B_1,\cdots,B_m$ be any
$(n\times n)$ real skew-symmetric matrices. Their coefficients under
the standard basis $\{\tilde{E}_{\alpha}\}_{\alpha\in S}$ of
$\mathfrak{o}(n)$ are determined by a matrix $B\in M(N,m)$ as
$(B_1,\cdots,B_m)=(\tilde{E}_1,\cdots,\tilde{E}_N)B$. Taking the
same ordering as in (\ref{order}) for $1\leq r<s\leq m$ and $1\leq
\alpha<\beta\leq N$, we arrange $\Big\{[B_r, B_s]\Big\}_{r<s}$,
$\Big\{[\tilde{E}_{\alpha}, \tilde{E}_{\beta}]\Big\}_{\alpha<\beta}$
into $C_m^2$, $C_N^2$-dimensional vectors respectively. We first observe that
\[([B_1, B_2],\cdots,[B_{m-1}, B_m])=([\tilde{E}_1, \tilde{E}_2],\cdots,[\tilde{E}_{N-1}, \tilde{E}_N])\cdot \varphi(B).\]

 Let $C(\tilde{E})$ denote the matrix in $M(C_N^2)$
defined by $C(\tilde{E})_{(\alpha,\beta)(\gamma,\tau)}:=\langle~
[\tilde{E}_{\alpha}, \tilde{E}_{\beta}],~ [\tilde{E}_{\gamma},
\tilde{E}_{\tau}]~ \rangle$, for $1\leq\alpha<\beta\leq N$,
$1\leq\gamma<\tau\leq N$. Moreover we will use the same notation for
$\{B_r\}$ and $\{\tilde{Q}_{\alpha}\}$, \emph{i.e.}, $C(B)$ and $C(Q)$
respectively. Then it is obvious that
\[C(B)=\varphi(B^t)C(\tilde{E})\varphi(B), \hskip 0.3cm C(Q)=\varphi(Q^t)C(\tilde{E})\varphi(Q).\]
Since $BB^t$ is a $(N\times N)$ semi-positive definite matrix, there
exists an orthogonal matrix $Q\in SO(N)$ such that $BB^t=Q~
diag(x_1,\cdots,x_N)~ Q^t$ with $x_{\alpha}\geq 0$, $1\leq\alpha\leq
N.$ Thus
\[\sum_{r=1}^m\|B_r\|^2=\|B\|^2 =\sum_{\alpha=1}^Nx_{\alpha}\] and
hence by Lemma \ref{lem6},
\begin{eqnarray}
\sum_{r,s=1}^m\|[B_r, B_s]\|^2&=&2Tr~C(B)=2Tr\hskip 0.1cm
\varphi(B^t)C(\tilde{E})\varphi(B)=2Tr~
\varphi(BB^t)C(\tilde{E})\nonumber\\&=&2Tr~
\varphi(diag(x_1,\cdots,x_N))C(Q)=\sum_{\alpha,\beta=1}^Nx_{\alpha}x_{\beta}\|[\tilde{Q}_{\alpha},
\tilde{Q}_{\beta}]\|^2.\nonumber
\end{eqnarray}

 We are now ready to prove Theorem \ref{thm1}. \\
\begin{bf}
 Proof of Theorem \ref{thm1}.
 \end{bf}
Put $d(n):= \frac{1}{3}$ if $n=3$ and $\frac{2}{3}$ if $n\geq4$. It
follows from the arguments above that the inequalities of the
theorem are equivalent to the following
\begin{equation}\label{ineq-poly}
\sum_{\alpha,\beta=1}^Nx_{\alpha}x_{\beta}\|[\tilde{Q}_{\alpha},
\tilde{Q}_{\beta}]\|^2\leq
d(n)\Big(\sum_{\alpha=1}^Nx_{\alpha}\Big)^2, ~~~~for~any~ x\in
\mathbb{R}^N_{+}, ~Q\in SO(N),
\end{equation}
where $\mathbb{R}^N_{+}:=\{0\neq
x=(x_1,...,x_N)\in\mathbb{R}^N~|~x_{\alpha}\geq0, 1\leq\alpha\leq
N\}$.

For $n=3$, $N=\frac{n(n-1)}{2}=3$ and by Lemma \ref{lem5}, we have
$\|[\tilde{Q}_{\alpha},\tilde{Q}_{\beta}]\|^2\leq\frac{1}{2}$ and
thus
\[\sum_{\beta\in S}\|[\tilde{Q}_{\alpha},\tilde{Q}_{\beta}]\|^2\leq\frac{1}{2}\times2=1.\]
On the other hand, it follows from Lemma \ref{lem4} that
$\sum_{\beta\in
S}\|[\tilde{Q}_{\alpha},\tilde{Q}_{\beta}]\|^2=n-2=1$. Therefore, we
get
\[\|[\tilde{Q}_{\alpha},\tilde{Q}_{\beta}]\|^2=\frac{1}{2},~~~~ for~any~ \alpha\neq\beta\in S.\]
In fact, this equality just says that the cross product of two
orthogonal unit vectors in $\mathbb{R}^3$ is still a unit vector if
we identify $\mathfrak{o}(3)$ with $\mathbb{R}^3$ and correspond the
commutator operator to the cross product. So in this case, the
inequality (\ref{ineq-poly}) is equivalent to
\[
x_1x_2+x_2x_3+x_3x_1\leq\frac{1}{3}(x_1+x_2+x_3)^2,~~~~for~any~ x\in
\mathbb{R}^3_{+},
\]
which is easily verified by
\[x_1x_2+x_2x_3+x_3x_1-\frac{1}{3}(x_1+x_2+x_3)^2=-\frac{1}{6}\Big((x_1-x_2)^2+(x_2-x_3)^2+(x_3-x_1)^2\Big)\leq0.\]
Note that the equality above holds if and only if
$x_1=x_2=x_3:=\lambda^2$, \emph{i.e.}, $BB^t=\lambda^2I_3$, which,
by Lemma \ref{lem7}, is equivalent to that there is a $R\in O(m)$
such that
\[(B_1,\cdots,B_m)=(\tilde{E}_{12},\tilde{E}_{13},\tilde{E}_{23})\cdot \Big(\lambda I_3,0_{3\times(m-3)}\Big)R=(C_1,C_2,C_3,0,\cdots,0)R.\]
This completes the proof of (i) of Theorem \ref{thm1}.

Now we consider the case (ii). Put
$$f_Q(x)=F(x,Q):=\sum_{\alpha,\beta=1}^Nx_{\alpha}x_{\beta}\|[\tilde{Q}_{\alpha},
\tilde{Q}_{\beta}]\|^2-\frac{2}{3}\Big(\sum_{\alpha=1}^Nx_{\alpha}\Big)^2.$$
Then $F$ is a continuous function defined on $\mathbb{R}^N\times
SO(N)$ and thus uniformly continuous on any compact subset of
$\mathbb{R}^N\times SO(N)$. Let
$\bigtriangleup:=\{x\in\mathbb{R}^N_{+}~|~\sum_{\alpha}x_{\alpha}=1\}$
and for any sufficiently small $\varepsilon>0$,
$\bigtriangleup_{\varepsilon}:=\{x\in
\bigtriangleup~|~x_{\alpha}\geq \varepsilon, 1\leq\alpha\leq N\}$.
Also let
$$G:=\{Q\in SO(N)~|~f_Q(x)\leq 0, ~for~all~ x\in
\bigtriangleup\},$$
$$G_{\varepsilon}:=\{Q\in SO(N)~|~f_Q(x)< 0,
~for~all~ x\in \bigtriangleup_{\varepsilon}\}.$$ We claim that
$G=\lim_{\varepsilon\rightarrow 0}G_{\varepsilon}=SO(N).$ Note that
this implies (\ref{ineq-poly}) and thus proves the inequality. In
fact we can show
\begin{equation}\label{G epsilon}
G_{\varepsilon}=SO(N) ~ for ~ any~ sufficiently ~ small ~
\varepsilon>0.
\end{equation}
To prove (\ref{G epsilon}), we use the continuity method, in which
we must prove the following three properties:
\begin{itemize}
\item[(a)]\label{step1} $I_N\in G_{\varepsilon}$ (and thus
$G_{\varepsilon}\neq\emptyset$);
\item[(b)]\label{step2} $G_{\varepsilon}$ is open in $SO(N)$;
\item[(c)]\label{step3} $G_{\varepsilon}$ is closed in $SO(N)$.
\end{itemize}

 Since $F$ is uniformly continuous on
$\triangle_{\varepsilon}\times SO(N)$, (b) is obvious.\\
\textbf{Proof of (a)}. For any $x\in \bigtriangleup_{\varepsilon}$,
$f_{I_N}(x)=\sum_{\alpha,\beta=1}^Nx_{\alpha}x_{\beta}\|[\tilde{E}_{\alpha},
\tilde{E}_{\beta}]\|^2-\frac{2}{3}\Big(\sum_{\alpha=1}^Nx_{\alpha}\Big)^2$.
\vskip 0.05cm It follows from (\ref{E tilde norm}) that
\begin{eqnarray}
f_{I_N}(x)&=&\sum_{i<j<k}(x_{ij}x_{jk}+x_{ij}x_{ik}+x_{ik}x_{jk})-\frac{2}{3}\Big(\sum_{i<j}x_{ij}\Big)^2\nonumber\\
&<&\sum_{i<j<k}(x_{ij}x_{jk}+x_{ij}x_{ik}+x_{ik}x_{jk})-\frac{2}{3}\sum_{i<j<k}2(x_{ij}x_{jk}+x_{ij}x_{ik}+x_{ik}x_{jk})\nonumber\\
&<&0,\nonumber
\end{eqnarray}
which means $I_N\in G_{\varepsilon}$. \hfill $\Box$\\
\textbf{Proof of (c)}. We only need to prove the following a priori
estimate: Suppose $f_Q(x)\leq 0$ for every $
x\in\bigtriangleup_{\varepsilon}$. Then $f_Q(x)< 0$ for every $
x\in\bigtriangleup_{\varepsilon}$.

The proof of this estimate is as follows: If there is a point $y\in
\bigtriangleup_{\varepsilon}$ such that $f_Q(y)= 0$, we can assume
without loss of generality that $$y\in
\bigtriangleup^{\gamma}_{\varepsilon}:=\{x\in\bigtriangleup_{\varepsilon}~|~x_{\alpha}>\varepsilon~
for~ \alpha\leq \gamma ~and~ x_{\beta}=\varepsilon~
for~\beta>\gamma\}$$ for some $1\leq \gamma\leq N$. Then $y$ is a
maximum point of $f_Q(x)$ in the cone spanned by
$\bigtriangleup_{\varepsilon}$ and an interior maximum point in
$\bigtriangleup^{\gamma}_{\varepsilon}$. Hence there exist numbers
$b_{\gamma+1},\cdots,b_N$ and a number $a$ such that
\begin{equation}\label{partial f}
\begin{array}{ll}
\Big(\frac{\partial f_Q}{\partial x_1}(y),\cdots,\frac{\partial
f_Q}{\partial x_{\gamma}}(y)\Big)=2a(1,\cdots,1),&\\
\Big(\frac{\partial f_Q}{\partial
x_{\gamma+1}}(y),\cdots,\frac{\partial f_Q}{\partial
x_{N}}(y)\Big)=2(b_{\gamma+1},\cdots,b_N)&
\end{array}
\end{equation}
or equivalently
\begin{equation}\label{partial f 2}
\sum_{\beta=1}^Ny_{\beta}(\|[\tilde{Q}_{\alpha},
\tilde{Q}_{\beta}]\|^2)-\frac{2}{3}=\Big\{
\begin{array}{ll}
 a & \alpha\leq\gamma,\\
b_{\alpha}& \alpha>\gamma.
\end{array}
\end{equation}
Hence
\[
f_Q(y)=\Big(\sum_{\alpha=1}^{\gamma}y_{\alpha}\Big)a+\Big(\sum_{\alpha=\gamma+1}^Nb_{\alpha}\Big)\varepsilon
=0\quad and \quad
\sum_{\alpha=1}^{\gamma}y_{\alpha}+(N-\gamma)\varepsilon=1.
\]
Meanwhile, we see $\frac{\partial f_Q}{\partial \nu}(y)=2(a\gamma
+\sum_{\alpha=\gamma+1}^Nb_{\alpha})\leq 0$, where
$\nu=(1,\cdots,1)$ is the vector normal to $\bigtriangleup$ in
$\mathbb{R}^N$. For any sufficiently small $\varepsilon$ (such as
$\varepsilon<1/N$), it follows from the above three formulas that
$a\geq 0$. Without loss of generality, we assume
$y_1=max\{y_1,\cdots,y_{\gamma}\}>\varepsilon$. Let $J:=\{\beta\in
S~|~ \|[\tilde{Q}_{1}, \tilde{Q}_{\beta}]\|^2\geq \frac{2}{3}\}$,
and let $n_1$ be the number of elements of $J$.
 Now combining Lemma \ref{lem3} , Lemma \ref{lem4} and Equation (\ref{partial f 2}) will give a
contradiction as follows:
\begin{eqnarray}
\frac{2}{3}\leq
\frac{2}{3}+a&=&\sum_{\beta=2}^Ny_{\beta}\|[\tilde{Q}_1,
\tilde{Q}_{\beta}]\|^2\nonumber\\
&=&\sum_{\beta\in J}y_{\beta}(\|[\tilde{Q}_1,
\tilde{Q}_{\beta}]\|^2-\frac{2}{3})+\frac{2}{3}\sum_{\beta\in
J}y_{\beta}+\sum_{\beta\in
S/J}y_{\beta}\|[\tilde{Q}_1, \tilde{Q}_{\beta}]\|^2\nonumber\\
&\leq&y_1\sum_{\beta\in J}(\|[\tilde{Q}_1,
\tilde{Q}_{\beta}]\|^2-\frac{2}{3})+\frac{2}{3}\sum_{\beta\in
J}y_{\beta}+\sum_{\beta\in S/J}y_{\beta}\|[\tilde{Q}_1,
\tilde{Q}_{\beta}]\|^2\nonumber\\
&\leq&\frac{2}{3}y_1+\frac{2}{3}\sum_{\beta\in
J}y_{\beta}+\sum_{\beta\in S/J}y_{\beta}\|[\tilde{Q}_1,
\tilde{Q}_{\beta}]\|^2 \leq
\frac{2}{3}\sum_{\beta=1}^Ny_{\beta}=\frac{2}{3}.\label{line}
\end{eqnarray}
Thus
\begin{equation}\label{n1N}
a=0\quad and \quad \sum_{\beta\in J}\|[\tilde{Q}_1,
\tilde{Q}_{\beta}]\|^2=\frac{2}{3}(n_1+1)\leq n-2<\frac{2}{3}N.
\end{equation}
Hence $S/(J\cup\{1\})\neq\emptyset$, and the second ``$\leq$" in
line (\ref{line}) should be ``$<$" by the definition of $J$ and the
positivity of $y_{\beta}$ for $\beta\in S/(J\cup\{1\})$.\hfill
$\Box$

Now we consider the equality condition of (ii) of Theorem \ref{thm1}
in view of the proof of the a priori estimate.

If there is an orthogonal matrix $Q$ and a point $y\in
\bigtriangleup$ such that $f_Q(y)= 0$, we can assume without loss of
generality that $$y\in
\bigtriangleup^{\gamma}:=\{x\in\bigtriangleup~|~x_{\alpha}>0~
for~all~\alpha\leq \gamma ~and~ x_{\beta}=0~
for~all~\beta>\gamma\}$$ for some $2\leq \gamma\leq N$. Then $y$ is
a maximum point of $f_Q(x)$ in $\mathbb{R}^N_{+}$ and an interior
maximum point in $\bigtriangleup^{\gamma}$. Therefore, we have the
same conclusions as (\ref{partial f}, \ref{partial f 2},
\ref{line}, \ref{n1N}) when $\gamma\leq n_1+1$, and all
inequalities in the proof of Lemma \ref{lem3} can be replaced by
equalities. So $n_0=1$ by Lemma \ref{lem2},
$\check{S}=\{(1,4),(2,3)\}$,
$\check{q}^{\beta}_{ii}=\check{q}^{\beta}_{ij}=0$ for any $(i,j)\in
S/\check{S}$, $\beta\in J$, which imply that $\tilde{Q}_{\beta}$ is
a linear combination of $diag(D_2,0)$, $diag(D_3,0)$ for any
$\beta\in J$. Hence, $1\leq n_1\leq 2$. But if $n_1=1$, it follows
from (\ref{n1N}) that
$\|[\tilde{Q}_1,\tilde{Q}_{\beta}]\|^2=\frac{4}{3}>1$ for $\beta\in
J$ which contradicts with Lemma \ref{lem5}. So we have $n_1=2$ and
$2\leq\gamma\leq3$. If $\gamma=2$, then it follows from Lemma
\ref{lem5} and (\ref{line}) the following contradiction:
$$\frac{2}{3}=y_2\|[\tilde{Q}_1, \tilde{Q}_2]\|^2\leq \frac{1}{2}.$$
So we get $\gamma=3$. By (\ref{line}) again, we have
$y_1=y_2=y_3=\frac{1}{3}$ and $y_\alpha=0$ for $\alpha>3$, and
$$\|[\tilde{Q}_1,\tilde{Q}_2]\|^2=\|[\tilde{Q}_1,\tilde{Q}_3]\|^2=\|[\tilde{Q}_2,\tilde{Q}_3]\|^2=1,$$
from which we can conclude the equality case of (ii) of Theorem
\ref{thm1} by Lemmas \ref{lem5} and \ref{lem7}. The proof of Theorem
\ref{thm1} is now completed.\hfill $\Box$

\section{Simons-type inequality for Riemannian submersions}\label{simons-sec}
\subsection{Moving frame method for Riemannian submersions} In this
subsection we present a treatment of basic materials about
Riemannian submersions by moving frame method.

Let $\pi: M^{n+m}\rightarrow B^n$ be a Riemannian submersion. We
denote by $D$, $R$, $r$ (resp. $\hat{D}, \hat{R}, \hat{r}$;
$\check{D},\check{R},\check{r}$) the Levi-Civita connection, the
curvature operator and the Ricci curvature on $M$ (resp. on the
fibres; on $B$) respectively. Around each point $x\in M$, we can
choose local orthonormal vertical vector fields
$\{U_{n+1},\cdots,U_{n+m}\}$ and local orthonormal basic vector
fields $\{X_1,\cdots,X_n\}$ which are horizontal and projectable
such that $\{\check{X}_1:=\pi_*X_1,\cdots,\check{X}_n:=\pi_*X_n \}$
form a local orthonormal basis around $\pi(x)\in B$. Thus
$\{X_1,\cdots,X_n,U_{n+1},\cdots,U_{n+m}\}$ form a local orthonormal
basis of $TM$ around $x\in M$ and we denote by
$\{\omega_1,\cdots,\omega_n,\omega_{n+1},\cdots,\omega_{n+m}\}$ the
dual $1$-forms on $M$ with respect to this basis, \emph{i.e.},
$$\omega_i(X_j)=\delta_{ij},\quad \omega_i(U_r)=\omega_r(X_i)=0,\quad \omega_r(U_s)=\delta_{rs},$$
where, from now on, we use the convention for indices as follows:
\begin{equation*}
h,i,j,k,l\in\{1,\cdots,n\};\quad
r,s,t,u,v\in\{n+1,\cdots,n+m\},\quad
\alpha,\beta,\gamma,\delta\in\{1,\cdots,n+m\}.
\end{equation*}
 Also we denote by
$\{\check{\omega}_1,\cdots,\check{\omega}_n\}$ the dual $1$-forms on
$B$ with respect to the basis $\{\check{X}_1,\cdots,\check{X}_n\}$
and by $\{\hat{\omega}_{n+1},\cdots,\hat{\omega}_{n+m}\}$ the dual
$1$-forms on the fibre(s) with respect to the basis
$\{U_{n+1},\cdots,U_{n+m}\}$. Then the connection $1$-forms
$\{\omega_{\alpha\beta}\}$ of $D$ on $M$, the connection $1$-forms
$\{\hat{\omega}_{rs}\}$ of $\hat{D}$ on the fibre(s) and the
connection $1$-forms $\{\check{\omega}_{ij}\}$ of $\check{D}$ can be
defined as follows:
\begin{equation}\label{connection forms}
\begin{array}{ll}
\omega_{ij}=\langle DX_i,X_j\rangle,&
\omega_{ir}=\langle DX_i,U_r\rangle=-\langle DU_r,X_i\rangle=-\omega_{ri},\quad
\omega_{rs}=\langle DU_r,U_s\rangle;\\
\hat{\omega}_{rs}=\langle\hat{D}U_r,U_s\rangle,&
\check{\omega}_{ij}=\langle\check{D}\check{X}_i,\check{X}_j\rangle,
\end{array}
\end{equation}
where without confusion we denote by bracket simultaneously the
metrics on $M$, $B$ and the fibres. Let $\{\Omega_{\alpha\beta}\}$
(resp. $\{\hat{\Omega}_{rs}\}$; $\{\check{\Omega}_{ij}\}$) be the
curvature $2$-forms on $M$ (resp. on the fibres; on $B$ ). Then we
have the following structure equations:
\begin{equation}\label{str-eq1}
\Big\{
\begin{array}{ll}
d\omega_{\alpha}=\omega_{\alpha\beta}\wedge\omega_{\beta},\quad
\omega_{\alpha\beta}=-\omega_{\beta\alpha},\\
d\omega_{\alpha\beta}=\omega_{\alpha\gamma}\wedge\omega_{\gamma\beta}+\Omega_{\alpha\beta};
\end{array}
\end{equation}
\begin{equation}\label{str-eq2}
\Big\{\begin{array}{ll}
d\hat{\omega}_{r}=\hat{\omega}_{rs}\wedge\hat{\omega}_{s},\quad
\hat{\omega}_{rs}=-\hat{\omega}_{sr},&\\
d\hat{\omega}_{rs}=\hat{\omega}_{rt}\wedge\hat{\omega}_{ts}+\hat{\Omega}_{rs};&
\end{array}
\end{equation}
\begin{equation}\label{str-eq3}
\Big\{\begin{array}{ll}
d\check{\omega}_{i}=\check{\omega}_{ij}\wedge\check{\omega}_{j},\quad
\check{\omega}_{ij}=-\check{\omega}_{ji},&\\
d\check{\omega}_{ij}=\check{\omega}_{ik}\wedge\check{\omega}_{kj}+\check{\Omega}_{ij},&
\end{array}
\end{equation}
where, from now on, repeated indices are implicitly summed over, and
we will write the curvature forms as
$\Omega_{\alpha\beta}=-\frac{1}{2}R_{\alpha\beta\gamma\delta}\omega_{\gamma}\wedge\omega_{\delta}$
and so the Ricci curvature $r=(R_{\alpha\beta})$ (resp.
$\hat{r}=(\hat{R}_{rs})$; $\check{r}=(\check{R}_{ij})$) on $M$
(resp. on the fibre(s); on $B$) can be expressed as
$R_{\alpha\beta}=R_{\alpha\gamma\beta\gamma}$ (resp.
$\hat{R}_{rs}=\hat{R}_{rtst}$; $\check{R}_{ij}=\check{R}_{ikjk}$).

Now since $\pi_*[X_i,U_r]=[\check{X}_i,\pi_*U_r]=0$ and
$\pi_*[U_r,U_s]=[\pi_*U_r,\pi_*U_s]=0$, $[X_i,U_r]$ and $[U_r,U_s]$
are vertical, thereby it follows from (\ref{connection forms}) and
the definitions of the tensors $T$ and $A$ in $(\ref{A T})$ that
\begin{equation}\label{ATcoeff}
\begin{array}{ll}
T^i_{rs}:=\omega_{ri}(U_s)=\langle T_{U_s}U_r,X_i\rangle=-\langle T_{U_s}X_i,U_r\rangle=T^i_{sr};&\\
A^r_{ij}:=\omega_{ir}(X_j)=\langle A_{X_j}X_i,U_r\rangle=-\langle A_{X_j}U_r,X_i\rangle=\omega_{ij}(U_r)=-A^r_{ji}.&
\end{array}
\end{equation}
Hence one can see that the tensor $T$ (or its coefficients
$\{T^i_{rs}\}$) is just the second fundamental form when it is
restricted to vertical vector fields along the fibre(s). Meanwhile,
we find that
\begin{equation*}
A_{X_i}X_j=-A_{X_j}X_i=\frac{1}{2}\mathscr{V}[X_i,X_j]
\end{equation*}
and thus
\begin{equation*}
A_{X}Y=\frac{1}{2}\mathscr{V}[X,Y],\quad for~~X,Y\in\mathscr{H},
\end{equation*}
which shows that $A$ measures the integrability of the horizontal
distribution $\mathscr{H}$ and so it is usually called the
\emph{integrability tensor} of $\pi$. By (\ref{A norm}) and
(\ref{ATcoeff}), we have
\begin{equation}\label{A norm2}
|A|^2=\sum_{r,i,j}(A^r_{ij})^2.
\end{equation}
Moreover, formulas (\ref{ATcoeff}) imply the following equations:
\begin{equation}\label{conn relation}
\begin{array}{ll}
\omega_{ir}=A^r_{ij}\omega_j-T^i_{rs}\omega_s,&\\
\omega_{ij}=\pi^*\check{\omega}_{ij}+A^r_{ij}\omega_r.&
\end{array}
\end{equation}
Define the covariant derivatives of $T^i_{rs}$ and $A^r_{ij}$ by
\begin{equation}\label{cov-T-A}
\begin{array}{ll}
DT^i_{rs}:=dT^i_{rs}+T^i_{ts}\omega_{tr}+T^i_{rt}\omega_{ts}+T^j_{rs}\omega_{ji}
=:T^i_{rsj}\omega_j+T^i_{rst}\omega_t,&\\
DA^r_{ij}:=dA^r_{ij}+A^r_{kj}\omega_{ki}+A^r_{ik}\omega_{kj}+A^s_{ij}\omega_{sr}
=:A^r_{ijk}\omega_k+A^r_{ijs}\omega_s.
\end{array}
\end{equation}
Then it is easily seen from (\ref{ATcoeff}) and (\ref{cov-T-A}) that
\begin{equation}
\begin{array}{ll}\label{cov-T-A2}
T^i_{rsj}=\langle(D_{X_j}T)_{U_s}U_r,X_i\rangle=T^i_{srj},&T^i_{rst}=\langle(D_{U_t}T)_{U_s}U_r,X_i\rangle=T^i_{srt},\\
A^r_{ijk}=\langle(D_{X_k}A)_{X_j}X_i,U_r\rangle=-A^r_{jik},&A^r_{ijs}=\langle(D_{U_s}A)_{X_j}X_i,U_r\rangle=-A^r_{jis},
\end{array}
\end{equation}
which are the only components of $DT$ and $DA$ that cannot be
recovered from $T$ and $A$ at a point (cf. \cite{Be,O}).
Taking deferential of (\ref{conn relation}) by using (\ref{cov-T-A})
and the structure equations (\ref{str-eq1}, \ref{str-eq3}) we get
\begin{equation}\label{d conn1}
(DA^r_{ij}+A^r_{ik}A^s_{jk}\omega_s+T^i_{rs}A^s_{jk}\omega_k)\wedge\omega_j
=(DT^i_{rs}-T^i_{rt}T^k_{ts}\omega_k)\wedge\omega_s+\Omega_{ir},
\end{equation}
\begin{eqnarray}\label{d conn2}
\Omega_{ij}&=&\pi^*\check{\Omega}_{ij}+(A^r_{ij}A^r_{kl}+A^r_{ik}A^r_{jl})\omega_k\wedge\omega_l\\
&&+(A^r_{ijk}-A^s_{ij}T^k_{sr}+A^s_{jk}T^i_{sr}+A^s_{ki}T^j_{sr})\omega_k\wedge\omega_r\nonumber\\
&&+(A^r_{ijs}+T^i_{ts}T^j_{tr}+A^s_{ik}A^r_{kj})\omega_s\wedge\omega_r.\nonumber
\end{eqnarray}
Recall that the O'Neill's formula $\{0\}$ in \cite{O} is just the Gauss equation
on the fibre(s) derived from the structure equations
(\ref{str-eq1}, \ref{str-eq2}) and can be written as
\begin{equation}\label{o0}
R_{rstu}=\hat{R}_{rstu}-T^i_{rt}T^i_{su}+T^i_{st}T^i_{ru}.
\end{equation}
Taking values of (\ref{d conn1}) on $U_s\wedge U_t$, $X_j\wedge U_s$
and of (\ref{d conn2}) on $U_s\wedge U_r$, $X_k\wedge U_r$ and
$X_k\wedge X_l$, respectively, we can get the O'Neill's formulas
$\{1,2,2',3,4\}$ in \cite{O} as follows:
\begin{eqnarray}
&&R_{irst}=T^i_{rts}-T^i_{rst},\label{o1}\\
&&R_{irjs}=T^i_{rsj}+A^r_{ijs}-T^i_{rt}T^j_{ts}+A^r_{ik}A^s_{jk},\label{o2}\\
&&R_{ijsr}=A^s_{ijr}-A^r_{ijs}+A^r_{ik}A^s_{kj}-A^s_{ik}A^r_{kj}+T^i_{tr}T^j_{ts}-T^i_{ts}T^j_{tr},\label{o2'}\\
&&R_{ijkr}=-A^r_{ijk}+A^s_{ij}T^k_{sr}-A^s_{jk}T^i_{sr}-A^s_{ki}T^j_{sr},\label{o3}\\
&&R_{ijkl}=\check{R}_{ijkl}\circ
\pi-2A^r_{ij}A^r_{kl}-A^r_{ik}A^r_{jl}+A^r_{il}A^r_{jk}.\label{o4}
\end{eqnarray}
 Taking value of (\ref{d
conn1}) on $X_j\wedge X_k$ we get
\begin{equation*}
R_{irjk}=A^r_{ijk}-A^r_{ikj}+2A^s_{jk}T^i_{rs},
\end{equation*}
which by combining with (\ref{cov-T-A2}, \ref{o3}) implies
\begin{equation}\label{Arijk-identity}
A^r_{ijk}+A^r_{jki}+A^r_{kij}=A^s_{ji}T^k_{sr}+A^s_{kj}T^i_{sr}+A^s_{ik}T^j_{sr}.
\end{equation}
Reversing $i$ and $j$, $r$ and $s$ in (\ref{o2}) and using
(\ref{cov-T-A2}) and the symmetry of the curvature operator, we can
get the following (cf. \cite{Be, Gr}):
\begin{equation}\label{Arijs-identity}
A^r_{ijs}+A^s_{ijr}=T^j_{rsi}-T^i_{rsj}.
\end{equation}

Let $\{K_{\alpha\beta}\}$ (resp. $\{\hat{K}_{rs}\}$;
$\{\check{K}_{ij}\}$) be the sectional curvatures of $M$ (resp. of
the fibre(s); of $B$). Then it follows from (\ref{o0}-\ref{o4}) that
\begin{equation}\label{sec-curv}
\begin{array}{lll}
K_{rs}=\hat{K}_{rs}+\sum_i\Big((T^i_{rs})^2-T^i_{rr}T^i_{ss}\Big),\\
K_{ir}=T^i_{rri}-\sum_s(T^i_{rs})^2+\sum_j(A^r_{ij})^2,\\
K_{ij}=\check{K}_{ij}\circ\pi-3\sum_r(A^r_{ij})^2,
\end{array}
\end{equation}
where, unusually, repeated indices are not summed over. If the fibres
are totally geodesic, \emph{i.e.}, $T=0$, then by
(\ref{o0}-\ref{o4}) we have the following identities about Ricci
curvatures:
\begin{equation}\label{ricci-curv}
\begin{array}{lll}
R_{ir}=A^r_{ikk}=-\langle\check{\delta}A(X_i),U_r\rangle,\\
R_{rs}=\hat{R}_{rs}+A^r_{ij}A^s_{ij},\\
R_{ij}=\check{R}_{ij}\circ\pi-2A^r_{ik}A^r_{jk}.
\end{array}
\end{equation}
Hence if $M$ is Einstein with totally geodesic fibres, then we have
\begin{equation}\label{einsteinM}
R_{ir}=A^r_{ikk}=-\langle\check{\delta}A(X_i),U_r\rangle=0,
\end{equation}
which is equivalent to that the horizontal distribution $\mathscr{H}$ is
Yang-Mills.
\subsection{Laplacians of the integrability tensor} From now on, we
assume that the Riemannian submersion $\pi: M^{n+m}\rightarrow B^n$
has totally geodesic fibres and Yang-Mills horizontal distribution,
\emph{i.e.}, $T=0$ and $A^r_{ikk}=0$ (by (\ref{einsteinM})).

We define the covariant derivatives of $A^r_{ijk}$ and $A^r_{ijs}$
by
\begin{equation}\label{Arijkl-Arijst}
\begin{array}{ll}
DA^r_{ijk}:=dA^r_{ijk}+A^r_{ljk}\omega_{li}+A^r_{ilk}\omega_{lj}+
A^r_{ijl}\omega_{lk}+A^s_{ijk}\omega_{sr}=:A^r_{ijkl}\omega_l+A^r_{ijks}\omega_s,\\
DA^r_{ijs}:=dA^r_{ijs}+A^r_{kjs}\omega_{ki}+A^r_{iks}\omega_{kj}+A^r_{ijt}\omega_{ts}
+A^t_{ijs}\omega_{tr}=:A^r_{ijsk}\omega_k+A^r_{ijst}\omega_t.
\end{array}
\end{equation}
The horizontal and vertical Laplacians of $A^r_{ij}$ are defined by
\begin{equation}\label{def-lap-A}
\triangle^{\mathscr{H}}A^r_{ij}:=A^r_{ijkk},\quad
\triangle^{\mathscr{V}}A^r_{ij}:=A^r_{ijss},
\end{equation}
while the horizontal and vertical Laplacians of a function $f\in
C^{\infty}(M)$ are defined by
\begin{equation}\label{def-lap}
\triangle^{\mathscr{H}}f:=(X_iX_i-D_{X_i}X_i)f,\quad
\triangle^{\mathscr{V}}f:=(U_sU_s-D_{U_s}U_s)f.
\end{equation}
It is easily seen that these Laplacians are well-defined and relate
to the Laplace-Beltrami operator $\triangle$ of $M$ by
\begin{equation*}
\triangle=\triangle^{\mathscr{H}}+\triangle^{\mathscr{V}}.
\end{equation*}
Moreover, since the fibres are totally geodesic,
$\triangle^{\mathscr{V}}$ is just the Laplace-Beltrami operator,
also denoted by $\triangle$, along any fibre $F_b$ when restricted
to actions on functions of $F_b$, \emph{i.e.},
\begin{equation*}
(\triangle^{\mathscr{V}}f)|_{F_b}=\triangle(f|_{F_b}),\quad
for~~any~~f\in C^{\infty}(M).
\end{equation*}
Therefore, if $M$ is closed, then for any function $f\in
C^{\infty}(M)$, we have
\begin{equation}\label{int-lap-0}
\int_M\triangle^{\mathscr{H}}f~dV_M=0,\quad
\int_M\triangle^{\mathscr{V}}f~dV_M=0.
\end{equation}

Taking differential of the second equation of (\ref{cov-T-A}) by
using (\ref{cov-T-A}, \ref{Arijkl-Arijst}) and the structure
equations (\ref{str-eq1}) we get
\begin{eqnarray}\label{dDArij}
&&DA^r_{ijk}\wedge\omega_k+DA^r_{ijs}\wedge\omega_s\\
&=&-(A^r_{hj}A^s_{hk}A^s_{il}+A^r_{ih}A^s_{hk}A^s_{jl}+A^r_{hl}A^s_{hk}A^s_{ij}+A^r_{ijs}A^s_{kl})\omega_k\wedge\omega_l\nonumber\\
&&-A^r_{ijl}A^s_{lk}\omega_k\wedge\omega_s+(A^r_{hj}\Omega_{hi}+A^r_{ih}\Omega_{hj}+A^s_{ij}\Omega_{sr}).\nonumber
\end{eqnarray}
Evaluating (\ref{dDArij}) on $X_k\wedge X_l$ and $U_s\wedge U_t$,
respectively, we obtain
\begin{eqnarray}
&&A^r_{ijlk}-A^r_{ijkl}\label{Arijkl-lk}\\
&=&-(A^r_{hj}A^s_{hk}A^s_{il}+A^r_{ih}A^s_{hk}A^s_{jl}+A^r_{hl}A^s_{hk}A^s_{ij}+2A^r_{ijs}A^s_{kl})\nonumber\\
&&+(A^r_{hj}A^s_{hl}A^s_{ik}+A^r_{ih}A^s_{hl}A^s_{jk}+A^r_{hk}A^s_{hl}A^s_{ij})\nonumber\\
&&-(A^r_{hj}R_{hikl}+A^r_{ih}R_{hjkl}+A^s_{ij}R_{srkl}),\nonumber
\end{eqnarray}
\begin{eqnarray}
A^r_{ijts}-A^r_{ijst}=-(A^r_{hj}R_{hist}+A^r_{ih}R_{hjst}+A^u_{ij}R_{urst}).\label{Arijst-ts}
\end{eqnarray}
Now since $T=0$ and $A^r_{ikk}=0$, by combining the identities
(\ref{ATcoeff}, \ref{cov-T-A2}, \ref{o2'}, \ref{o4}, \ref{Arijk-identity}, \ref{Arijs-identity}, \ref{def-lap-A})
with (\ref{Arijkl-lk}, \ref{Arijst-ts}), we can calculate the
Laplacians of the integrability tensor $A$ as follows:
\begin{eqnarray}\label{lap-h-A}
&&\langle A,\triangle^{\mathscr{H}}A\rangle:=A^r_{ij}(\triangle^{\mathscr{H}}A^r_{ij})=A^r_{ij}A^r_{ijkk}\\
&=&A^r_{ij}(-A^r_{jkik}-A^r_{kijk})=2A^r_{ij}A^r_{ikjk}=2A^r_{ij}(A^r_{ikjk}-A^r_{ikkj})\nonumber\\
&=&2A^r_{ij}\Big(-(A^r_{hk}A^s_{hk}A^s_{ij}+2A^r_{iks}A^s_{kj})+2A^r_{hk}A^s_{hj}A^s_{ik}\nonumber\\
&&\quad\quad\quad-(A^r_{hk}R_{hikj}+A^r_{ih}R_{hkkj}+A^s_{ik}R_{srkj})\Big)\nonumber\\
&=&2A^r_{ij}\Big(2A^r_{ih}A^s_{hk}A^s_{kj}+2A^r_{hk}A^s_{hj}A^s_{ik}\nonumber\\
&&\quad\quad-(A^r_{hk}\check{R}_{hikj}\circ\pi+A^r_{ih}\check{R}_{hkkj}\circ\pi)
-2A^s_{ik}R_{srkj}\Big)\nonumber\\
&=&-2\|[A^r,A^s]\|^2-A^r_{ij}A^r_{hk}\check{R}_{ijhk}\circ\pi+2A^r_{ij}A^r_{ih}\check{R}_{jh}\circ\pi-4A^r_{ij}A^s_{ik}R_{srkj},\nonumber
\end{eqnarray}
\begin{eqnarray}\label{lap-v-A}
&&\langle A,\triangle^{\mathscr{V}}A\rangle:=A^r_{ij}(\triangle^{\mathscr{V}}A^r_{ij})=A^r_{ij}A^r_{ijss}\\
&=&-A^r_{ij}A^s_{ijrs}=A^r_{ij}(A^s_{ijsr}-A^s_{ijrs})\nonumber\\
&=&-A^r_{ij}(A^s_{hj}R_{hirs}+A^s_{ih}R_{hjrs}+A^u_{ij}R_{usrs})\nonumber\\
&=&2A^r_{ij}A^s_{ik}R_{srkj}-A^r_{ij}A^s_{ij}\hat{R}_{rs},\nonumber
\end{eqnarray}
where we denote by $A^r:=(A^r_{ij})$ the $(n\times n)$
skew-symmetric matrix corresponding to the operator
$AU_r:~TM\rightarrow TM$ defined by
$AU_r(X_i):=A_{X_i}U_r=A^r_{ij}X_j$, and the square norm of the Lie
bracket in the last line of (\ref{lap-h-A}) is implicitly summed
over all the indices $r$ and $s$.

\subsection{Simons-type inequality}In this subsection we will derive
the Simons-type inequality rendered in Theorem \ref{Thm-simonstype
ineq} for Riemannian submersions with totally geodesic fibres and
Yang-Mills horizontal distributions.

We denote by $\nabla^{\mathscr{H}}$ (resp. $\nabla^\mathscr{V}$) the
restriction to the horizontal (resp. vertical) distribution of the
covariant derivative $D$ on $M$, \emph{i.e.},
$$\nabla^{\mathscr{H}}W:=(DW)|_{\mathscr{H}},\quad \nabla^\mathscr{V}W:=(DW)|_{\mathscr{V}},\quad for~~ any~~ tensor~~ W~~
on~~ M.$$ From (\ref{A norm2}, \ref{def-lap-A}, \ref{def-lap}) we can
derive the following
\begin{equation}\label{lap-hv-Anorm}
%\begin{array}{ll}
\frac{1}{2}\triangle^{\mathscr{H}}|A|^2=\langle A,\triangle^{\mathscr{H}}A\rangle+|\nabla^{\mathscr{H}}A|^2,\quad
\frac{1}{2}\triangle^{\mathscr{V}}|A|^2=\langle A,\triangle^{\mathscr{V}}A\rangle+|\nabla^{\mathscr{V}}A|^2.
%\end{array}
\end{equation}
Combining (\ref{o2'}, \ref{o3}, \ref{Arijs-identity}, \ref{Arijkl-lk}, \ref{Arijst-ts}, \ref{lap-hv-Anorm})
we obtain
\begin{eqnarray}\label{laph-lapv}
&&(\frac{1}{2}\triangle^{\mathscr{H}}+2\triangle^{\mathscr{V}})|A|^2\\
&=&-2\|[A^r,A^s]\|^2-A^r_{ij}A^r_{hk}\check{R}_{ijhk}\circ\pi+2A^r_{ij}A^r_{ih}\check{R}_{jh}\circ\pi-4A^r_{ij}A^s_{ij}\hat{R}_{rs}\nonumber\\
&&+4A^r_{ij}A^s_{ik}R_{srkj}+|A^r_{ijk}|^2+4|A^r_{ijs}|^2\nonumber\\
&=&-\|[A^r,A^s]\|^2-A^r_{ij}A^r_{hk}\check{R}_{ijhk}\circ\pi+2A^r_{ij}A^r_{ih}\check{R}_{jh}\circ\pi-4A^r_{ij}A^s_{ij}\hat{R}_{rs}\nonumber\\
&&+|R_{ijkr}|^2+|R_{srij}|^2,\nonumber
\end{eqnarray}
where, from now on, the indices within square norms are also implicitly summed over.
If $M$ is closed, then by (\ref{int-lap-0}, \ref{laph-lapv}) we get
\begin{equation}\label{int-ineq}
\int_M\Big(\|[A^r,A^s]\|^2+4A^r_{ij}A^s_{ij}\hat{R}_{rs}+A^r_{ij}A^r_{hk}\check{R}_{ijhk}\circ\pi-2A^r_{ij}A^r_{ih}\check{R}_{jh}\circ\pi\Big)dV_M\geq0.
\end{equation}

As defined before Theorem \ref{Thm-simonstype ineq} in Section \ref{sec1}, for $x\in M$, $\check{\kappa}(x)$ is the largest eigenvalue of
the curvature operator $\check{R}$ of $B$ at $\pi(x)\in B$,
$\check{\lambda}(x)$ is the lowest eigenvalue of the Ricci curvature
$\check{r}$ of $B$ at $\pi(x)\in B$ and $\hat{\mu}(x)$ is the
largest eigenvalue of the Ricci curvature $\hat{r}$ of the fibre at
$x$. Then the inequality (\ref{int-ineq}) induces the following:
\begin{equation}\label{int-ineq2}
\int_M\Big(\|[A^r,A^s]\|^2+4\hat{\mu}|A|^2+2\check{\kappa}|A|^2-2\check{\lambda}|A|^2\Big)dV_M\geq0.
\end{equation}

When $n=2$, it is obvious that $[A^r,A^s]=0$ and
$\check{\kappa}=\check{\lambda}$.
Thus by (\ref{int-ineq2}) we have
\begin{equation*}\label{ineq-mu }
\int_M|A|^2\hat{\mu}~dV_M\geq0,
\end{equation*}
 which verifies the first case (i) of Theorem
\ref{Thm-simonstype ineq}.

When $m=1$, the first two terms of
(\ref{int-ineq2}) vanish and thus
\begin{equation*}\label{ineq-kappa-lam}
  \int_M|A|^2(\check{\kappa}-\check{\lambda})~dV_M\geq0,
\end{equation*}
which proves the second case (ii) of Theorem \ref{Thm-simonstype
ineq}.

The last two cases (iii, iv) of Theorem \ref{Thm-simonstype ineq} can be derived immediately by applying the inequalities (i, ii) of Theorem \ref{thm1} to (\ref{int-ineq2}) respectively. This interaction originally occurs between the DDVV inequality (\ref{DDVV sym}) and the Simons integral inequality as we mentioned in the introduction.

\subsection{Equality conclusions}
In this subsection we will complete the proof of Theorem
\ref{Thm-simonstype ineq} by verifying the conclusions (a-d) for
equality conditions of the Simons-type inequality case by case.

Firstly, it is a well-known fact that the total space $M$ of a
Riemannian submersion with vanishing $T$ and $A$ is (at least
locally) a Riemannian product $B\times F$, and vice versa.
Henceforth, we assume that $A\neq0$. The proof of (a-d) of Theorem
\ref{Thm-simonstype ineq} goes on as follows:

\textbf{(a)} In each case of (i-iv) of Theorem
\ref{Thm-simonstype ineq}, the equality assumption of the integral inequality compels (\ref{int-ineq}) to attain its equality simultaneously,
which then by (\ref{int-lap-0}, \ref{laph-lapv}) shows immediately
\begin{equation}\label{Rijkr-Rsrij-0}
R_{ijkr}\equiv0,\quad R_{srij}\equiv0.
\end{equation}
Now since the fibres are totally geodesic, the Ricci equation on any
fibre $F_b$ shows that the normal curvature $\hat{R}^{\bot}_{srij}$
of $F_b$ equals $R_{srij}$ and thus vanishes. So each fibre has flat
normal bundle in $M$. Moreover, it follows from
(\ref{o2'}, \ref{o3}, \ref{Arijs-identity}, \ref{Rijkr-Rsrij-0}) that
\begin{equation}\label{Arijk0-Arijs}
A^r_{ijk}=0,\quad A^r_{ijs}=\frac{1}{2}[A^r,A^s]_{ij}.
\end{equation}
Noticing that the covariant derivative of $|A|^2$ can be calculated
from (\ref{Arijk0-Arijs}) as
\begin{equation*}
D|A|^2=2A^r_{ij}A^r_{ijk}\omega_k+2A^r_{ij}A^r_{ijs}\omega_s=0,
\end{equation*}
we arrive at the conclusion that $|A|^2\equiv Const=:C>0$. Then by
(\ref{lap-h-A}-\ref{lap-hv-Anorm}) and
(\ref{Rijkr-Rsrij-0}, \ref{Arijk0-Arijs}), we have
\begin{eqnarray}
&&\frac{1}{2}\triangle^{\mathscr{H}}|A|^2=
-2\|[A^r,A^s]\|^2-A^r_{ij}A^r_{hk}\check{R}_{ijhk}\circ\pi+2A^r_{ij}A^r_{ih}\check{R}_{jh}\circ\pi\equiv0,\label{laph-0}\\
&&\frac{1}{2}\triangle^{\mathscr{V}}|A|^2=
-A^r_{ij}A^s_{ij}\hat{R}_{rs}+\frac{1}{4}\|[A^r,A^s]\|^2\equiv0.\label{lapv-0}
\end{eqnarray}
Now we come to prove the subcases (a1-a4) of (a) as follows.
\begin{itemize}
\item[(a1)] Now $n=2$ and $[A^r,A^s]\equiv0$. So by the
definition of $\hat{\mu}$ and (\ref{lapv-0}), we get
\begin{equation*}
|A|^2\hat{\mu}\geq A^r_{ij}A^s_{ij}\hat{R}_{rs}=0,
\end{equation*}
whereas $|A|^2\equiv C>0$ and $\int_M|A|^2\hat{\mu}dV_M=0$ by
assumption. \\This proves that $\hat{\mu}\equiv0$.
\item[(a2)] Now $m=1$ and $[A^r,A^s]\equiv0$. So by the
definitions of $\check{\kappa}$, $\check{\lambda}$  and
(\ref{laph-0}), we get
\begin{equation*}
|A|^2(\check{\kappa}-\check{\lambda})\geq
\frac{1}{2}A^r_{ij}A^r_{hk}\check{R}_{ijhk}\circ\pi-A^r_{ij}A^r_{ih}\check{R}_{jh}\circ\pi=0,
\end{equation*}
whereas $|A|^2\equiv C>0$ and
$\int_M|A|^2(\check{\kappa}-\check{\lambda})dV_M=0$ by assumption.
\\This proves that $\check{\kappa}-\check{\lambda}\equiv0$.
\item[(a3)] Now the equality assumption implies that the inequality
in (i) of Theorem \ref{thm1} (with $B_r=A^r$) attains its equality,
\emph{i.e.},
\begin{equation}\label{equ-Lie}
\sum_{r,s}\|[A^r,A^s]\|^2=\frac{1}{3}\Big(\sum_r|A^r|^2\Big)^2=\frac{1}{3}|A|^4=\frac{1}{3}C^2.
\end{equation}
Then by the definitions of $\hat{\mu},\check{\kappa},\check{\lambda}$ and (\ref{laph-0}, \ref{lapv-0}), we have
\begin{eqnarray}
&&|A|^2\hat{\mu}\geq
A^r_{ij}A^s_{ij}\hat{R}_{rs}=\frac{1}{4}\|[A^r,A^s]\|^2=\frac{1}{12}C^2,\nonumber\\
&&|A|^2(\check{\kappa}-\check{\lambda})\geq
\frac{1}{2}A^r_{ij}A^r_{hk}\check{R}_{ijhk}\circ\pi-A^r_{ij}A^r_{ih}\check{R}_{jh}\circ\pi=-\|[A^r,A^s]\|^2=-\frac{1}{3}C^2,\nonumber
\end{eqnarray}
whereas $|A|^2\equiv C>0$ and
$\int_M~|A|^2(\frac{1}{6}|A|^2+2\hat{\mu}+\check{\kappa}-\check{\lambda})~dV_M=0$
by assumption.\\
This proves that $\hat{\mu}\equiv\frac{1}{12}C$,
$\check{\kappa}-\check{\lambda}\equiv-\frac{1}{3}C$.
\item[(a4)] The proof is almost the same with that of (a3) except
for that the coefficient $\frac{1}{3}$ in (\ref{equ-Lie}) would be substituted by $\frac{2}{3}$. So we omit it
here.
\end{itemize}

\textbf{(b)} If the equality in (iii) (resp. (iv)) holds, as in the
proof of (a3), the inequality in (i) (resp. (ii)) of Theorem
\ref{thm1} (with $B_r=A^r$) attains its equality, thereby, under
some $K=O(n)\times O(m)$ action which can be realized by a choice of
an orthonormal horizontal basis $\{X_1,\cdots,X_n\}$ and of an
orthonormal vertical basis $\{U_{n+1},\cdots,U_{n+m}\}$, the
matrices $A^r$'s are all equal to zero except
$A^{n+1},A^{n+2},A^{n+3}$, which are in the forms of $C_1,C_2,C_3$
(resp. $diag(D_1, 0), diag(D_2, 0), diag(D_3, 0)$). Noticing that
now we have $$|A|^2=|A^{n+1}|^2+|A^{n+2}|^2+|A^{n+3}|^2\equiv C>0,$$
we derive that $m\geq3$. Moreover, we can rewrite
$A^{n+1},A^{n+2},A^{n+3}$ as follows:
\begin{equation}\label{case3A123}
\begin{array}{cc}
A^{n+1}=\sqrt{\frac{C}{6}}\left(\begin{array}{ccc}
0& 1 & 0 \\
-1& 0 & 0\\
 0& 0& 0
\end{array}\right),& A^{n+2}=\sqrt{\frac{C}{6}}\left(\begin{array}{ccc}
0& 0 & 1 \\
0& 0 & 0\\
-1& 0& 0
\end{array}\right),\\
 A^{n+3}=\sqrt{\frac{C}{6}}\left(\begin{array}{ccc}
0& 0 & 0 \\
0& 0 & 1\\
0& -1& 0
\end{array}\right)& \emph{for equality case of (iii);}
\end{array}
\end{equation}

\begin{equation}\label{case4A123}
\begin{array}{cc}
A^{n+1}=\sqrt{\frac{C}{12}}\left(\begin{array}{c|c}
\begin{smallmatrix}
0& 1 & 0&0 \\
-1& 0 & 0&0\\
 0& 0& 0&1\\
 0&0&-1&0
\end{smallmatrix}&0\\
 \hline0&0
\end{array}\right),& A^{n+2}=\sqrt{\frac{C}{12}}\left(\begin{array}{c|c}
\begin{smallmatrix}
0& 0 & 1 &0\\
0& 0 & 0&-1\\
-1& 0& 0&0\\
0&1&0&0
\end{smallmatrix}&0\\
 \hline0&0
\end{array}\right),\\
A^{n+3}=\sqrt{\frac{C}{12}}\left(\begin{array}{c|c}
\begin{smallmatrix}
0& 0 & 0 &1\\
0& 0 & 1&0\\
0& -1& 0&0\\
-1&0&0&0
\end{smallmatrix}&0\\
 \hline0&0
\end{array}\right)& \emph{for
equality case of (iv),}
\end{array}
\end{equation}
where $0$ in the diagonals of (\ref{case4A123}) is a zero matrix of order $(n-4)$. As in
the proof of (a3), we have the following equations if the equality
in (iii) or (iv) holds:
\begin{equation}\label{case34-equs}
|A|^2\hat{\mu}= A^r_{ij}A^s_{ij}\hat{R}_{rs},\quad
|A|^2\check{\kappa}=
\frac{1}{2}A^r_{ij}A^r_{hk}\check{R}_{ijhk}\circ\pi,\quad
|A|^2\check{\lambda}=A^r_{ij}A^r_{ih}\check{R}_{jh}\circ\pi.
\end{equation}

Using the formulas (\ref{case3A123}) for equality case of (iii), the
equations (\ref{case34-equs}) can be turned to the following:
\begin{equation*}
\begin{array}{lll} \hat{\mu}=
\frac{1}{3}(\hat{R}_{n+1~n+1}+\hat{R}_{n+2~n+2}+\hat{R}_{n+3~n+3}),\\
\check{\kappa}=
\frac{1}{3}(\check{R}_{1212}\circ\pi+\check{R}_{1313}\circ\pi+\check{R}_{2323}\circ\pi),\\
\check{\lambda}=\frac{1}{3}(\check{R}_{11}\circ\pi+\check{R}_{22}\circ\pi+\check{R}_{33}\circ\pi).
\end{array}
\end{equation*}
Then recalling the definitions of
$\hat{\mu},\check{\kappa},\check{\lambda}$, we obtain the following
decompositions for $\hat{r},\check{R},\check{r}$ for equality case
of (iii):
\begin{equation*}
\hat{r}=\hat{\mu}I_3\oplus \hat{r}',\quad
\check{R}\equiv\check{\kappa}I_3,\quad
\check{r}\equiv\check{\lambda}I_3,
\end{equation*}
where $\hat{r}'=\hat{r}|_{span\{U_{7},\cdots,U_{3+m}\}}$ if $m\geq4$
and $0$ if $m=3$, $\check{\lambda}=2\check{\kappa}$ because of $n=3$
now.

Similarly, using the formulas (\ref{case4A123}) for equality case of
(iv) and the first Bianchi identity, the equations
(\ref{case34-equs}) can be turned to the following:
\begin{equation*}
\begin{array}{lll} \hat{\mu}=
\frac{1}{3}(\hat{R}_{n+1~n+1}+\hat{R}_{n+2~n+2}+\hat{R}_{n+3~n+3}),\\
\check{\kappa}=
\frac{1}{6}(\check{R}_{1212}\circ\pi+\check{R}_{1313}\circ\pi+\check{R}_{1414}\circ\pi
+\check{R}_{2323}\circ\pi+\check{R}_{2424}\circ\pi+\check{R}_{3434}\circ\pi),\\
\check{\lambda}=\frac{1}{4}(\check{R}_{11}\circ\pi+\check{R}_{22}\circ\pi+\check{R}_{33}\circ\pi+\check{R}_{44}\circ\pi).
\end{array}
\end{equation*}
Then recalling the definitions of
$\hat{\mu},\check{\kappa},\check{\lambda}$, we obtain the following
decompositions for $\hat{r},\check{R},\check{r}$ for equality case
of (iv):
\begin{equation*}
\hat{r}=\hat{\mu}I_3\oplus \hat{r}',\quad
\check{R}=\check{\kappa}I_6\oplus \check{R}',\quad
\check{r}\equiv\check{\lambda}I_4\oplus\check{r}',
\end{equation*}
where $\hat{r}'=\hat{r}|_{span\{U_{n+4},\cdots,U_{n+m}\}}$ if
$m\geq4$ and $0$ if $m=3$, $\check{R}'=\check{R}|_{span\{X_i\wedge
X_j|1\leq i\leq n,~5\leq j\leq n\}}$ and
$\check{r}'=\check{r}|_{span\{X_5,\cdots,X_n\}}$ if $n\geq5$ and $0$
if $n=4$.

From the decompositions, if $m=3$, then we can see that the
$3$-dimensional fibres have constant Ricci curvature and thus have
constant sectional curvature; if $n=3$ or $4$, then the base
manifold $B^n$ has constant sectional curvature; if
$n=5$, then by the definitions of $\check{\kappa},\check{\lambda}$ we
have
\begin{eqnarray}
&&\check{\lambda}\leq\check{R}_{55}=\check{R}_{1515}+\check{R}_{2525}+\check{R}_{3535}+\check{R}_{4545}\leq
3\check{\kappa}+\check{R}_{i5i5},\nonumber\\
&&\check{\lambda}=\check{R}_{ii}=\sum_{j=1}^5\check{R}_{ijij}=3\check{\kappa}+\check{R}_{i5i5},\quad
for~~ i=1,2,3,4.\nonumber
\end{eqnarray}
These prove that $\check{R}_{i5i5}=\check{\kappa}$ for $i=1,2,3,4$,
and so the base manifold $B^5$ has constant sectional curvature.

\textbf{(c)} Now $m=3,n=3$ and the equality in (iii) holds. In (b)
we have proved that both of the fibres and the base manifold $B^3$
have constant sectional curvature. Due to a result of Hermann
\cite{He} we see that the fibres are all isometric. Reset $|A|^2\equiv C=:24a>0$, then by (a3) and (b) we get
$$\hat{\mu}=2a,\quad \check{\lambda}=2\check{\kappa}=16a,$$
which deduce the conclusions of (c1) and (c2).

The identities in (c3) can be
calculated from the formulas (\ref{sec-curv}, \ref{ricci-curv}, \ref{case3A123}). In fact, since we have $T=0$ and $A_{ikk}=0$, the formulas
(\ref{sec-curv}, \ref{ricci-curv}) turn into the following:
\begin{equation}\label{K_MR_M}
\begin{array}{ll}
K_{rs}=\hat{K}_{rs},\quad K_{ir}=\sum_j(A^r_{ij})^2,\quad K_{ij}=\check{K}_{ij}\circ\pi-3\sum_r(A^r_{ij})^2;\\
R_{ir}=0,\quad R_{rs}=\hat{R}_{rs}+A^r_{ij}A^s_{ij},\quad R_{ij}=\check{R}_{ij}\circ\pi-2A^r_{ik}A^r_{jk}.
\end{array}
\end{equation}
Then using formulas (\ref{case3A123}, \ref{K_MR_M}) and the known facts that $\hat{K}_{rs}=a$, $\check{K}_{ij}=8a$, we complete the proof.
One should notice that the index range for $r$ in (c3) is $\{1,2,3\}$ rather than $\{n+1,n+2,n+3\}$ $(n=3)$ here.

\textbf{(d)} Based on results of (b) and formulas (\ref{case4A123}, \ref{K_MR_M}),
the proof of the assertions for (d2) and the heading paragraph of
(d) are exactly the same with that of (c) despite that we reset
$|A|^2\equiv C=:12a>0$ here in view of (a4). As for (d1), we first
calculate the sectional curvatures of $B^4$ and $M^7$ respectively
and find that $B$ has constant sectional curvature $4a$ and $M$ has
constant sectional curvature $a$. In fact, by (a4), (b) and (\ref{case4A123}, \ref{K_MR_M}) we know that
$$\hat{\mu}=2a,\quad \check{\lambda}=3\check{\kappa}=12a,\quad K_{rs}=K_{ir}=K_{ij}=a.$$

Hence, $M^7$ is covered by $S^7(\frac{1}{\sqrt{a}})$, $B^4$ is covered by $S^4(\frac{1}{2\sqrt{a}})$ and we denote by $\pi_1, \pi_2$ the corresponding covering maps. Thus there is a Riemannian submersion $\pi_0: S^7(\frac{1}{\sqrt{a}})\rightarrow S^4(\frac{1}{2\sqrt{a}})$ (lift map of $\pi\circ\pi_1$ through $\pi_2$) such that $\pi_2\circ\pi_0=\pi\circ\pi_1$. Recall that Ranjan \cite{Ra} showed that $\pi_0: S^7(\frac{1}{\sqrt{a}})\rightarrow S^4(\frac{1}{2\sqrt{a}})$ is equivalent to the Hopf fibration (see also \cite{Es}). Without loss of generality, we can assume that $\pi_0$ is just the Hopf fibration, since otherwise we can alter $\pi_1, \pi_2$ by taking compositions with corresponding isometries (bundle isometry between $\pi_0$ and the Hopf fibration) of $S^7(\frac{1}{\sqrt{a}})$ and $S^4(\frac{1}{2\sqrt{a}})$ respectively. The proof of (d1) is now completed.

In conclusion, the proof of Theorem \ref{Thm-simonstype ineq} is now completed.

%%%%%%%%%%%%%%%%%%%%%%%%%%%%%%%%%%%%%%%%%%%%%%%%%%%%%%%%%%%%%%%%%%%%%%%%%%%%%%%%%%%%%%%%%%%%%%%%%
\begin{ack}
I would like to thank Professors Thomas E. Cecil, Qingming Chen and
Weiping Zhang for their kindly encouragements and supports. Many thanks also to Professors Xiuxiong Chen, Zhiqin Lu and Yibin Shen for their useful suggestions and discussions.
\end{ack}

\end{document}